\documentclass[12pt]{amsart}
\usepackage[T2A]{fontenc}
\usepackage[cp1251]{inputenc}
\usepackage[english]{babel}
\usepackage{graphics}

\usepackage{amsfonts,amssymb,mathrsfs,amscd,comment}

\usepackage{hyperref} % added by A.G.to create references in pdf.
\usepackage{url}      % added by A.G.

\usepackage{graphicx,euscript,amsthm}
\usepackage[all]{xy}
\usepackage[in]{fullpage}
\usepackage{pstricks,pst-grad,pst-node}

\theoremstyle{plain}
\newtheorem{theorem}{Theorem}[section]
\newtheorem{lemma}[theorem]{Lemma}
\newtheorem{proposition}[theorem]{Proposition}
\newtheorem{corollary}[theorem]{Corollary}
\theoremstyle{definition}

\theoremstyle{remark}
\newtheorem{remark}[theorem]{Remark}

\renewcommand{\int}{\mathrm{int}}

\renewcommand{\epsilon}{\varepsilon}
\renewcommand{\vert}{\mathrm{vert }}

\renewcommand{\phi}{\varphi}

\makeatletter
\let\@fnsymbol\@arabic
\makeatother

\author{Igor A. Baburin}\thanks{Theoretische Chemie, Technische Universit{\"a}t Dresden, 01062 Dresden, Germany, Igor.Baburin@chemie.tu-dresden.de}

\author{Mikhail Bouniaev}\thanks{School of Mathematical \& Statistical Sciences, University of Texas Rio Grande Valley, Brownsville, TX 78520, USA, mikhail.bouniaev@utrgv.edu}

\author{Nikolay Dolbilin}\thanks{Steklov Mathematical Institute, Moscow 117966, Russia, dolbilin@mi.ras.ru}

\author{Nikolay Yu. Erokhovets}\thanks{Steklov Mathematical Institute, Moscow 117966, Russia, erochovetsn@hotmail.com}

\author{Alexey Garber}\thanks{School of Mathematical \& Statistical Sciences, University of Texas Rio Grande Valley, Brownsville, TX 78520, USA, alexeygarber@gmail.com}

\author{Sergey V. Krivovichev}\thanks{Kola Science Centre of the Russian Academy of Sciences, Apatity 184209, Russia and Department of Crystallography, Saint Petersburg State University, Saint Petersburg 199155, Russia, s.krivovichev@spbu.ru}

\author{Egon Schulte}\thanks{Department of Mathematics, Northeastern University, Boston, MA 02115, USA, schulte@neu.edu}

\title{On the Origin of Crystallinity: a Lower Bound for the Regularity Radius of Delone Sets}

\begin{document}

\begin{abstract}
The mathematical conditions for the origin of long-range order or crystallinity in ideal crystals is one of the very fundamental problems of modern crystallography. It is widely believed that the (global) regularity of crystals is a consequence of 'local order', in particular the repetitivity of local fragments, but the exact mathematical theory of this phenomena is poorly known. In particular, most mathematical models for quasicrystals, for example Penrose tiling, have repetitive local fragments, but are not (globally) regular. The universal abstract models of any atomic arrangements are Delone sets, which are uniformly distributed discrete point sets in Euclidean $d$-space. An ideal crystal is a regular or multi-regular system, that is, a Delone set, which is the orbit of a single point or finitely many points under a crystallographic group of isometries. The local theory of regular or multi-regular systems aims at finding sufficient local conditions for a Delone set $X$ to be a regular or multi-regular system. One of the main goals is to estimate the {\it regularity radius\/} $\hat{\rho}_d$ for Delone sets $X$ in terms of the radius $R$ of the largest ``empty ball'' for $X$. The celebrated ``local criterion for regular systems'' provides an upper bound for $\hat{\rho_d}$ for any $d$. Better upper bounds are known for $d\leq 3$. The present paper establishes the lower bound $\hat{\rho_d}\geq 2dR$ for all $d$, which is linear in $d$. The best previously known lower bound had been $\hat{\rho}_d\geq 4R$ for $d\geq 2$. The proof of the new lower bound is accomplished through explicit constructions of Delone sets with mutually equivalent $(2dR-\varepsilon)$-clusters, which are not regular systems. The $2$- and $3$-dimensional constructions are illustrated by examples. In addition to its fundamental importance, the obtained result is also relevant for the understanding of geometrical conditions of formation of ordered and disordered arrangements in polytypic materials.
\end{abstract}

\maketitle                 
     
\section{Introduction}
\label{intro}

The standard mathematical model of an ideal crystal is based on two fundamental concepts:\ 
a uniformly distributed discrete point set (or Delone set) and a crystallographic group. An ideal crystal structure is modeled by a Delone set that can be split into crystallographic orbits, i.e. the Delone subsets invariant under some crystallographic group.

The origin of crystallinity, i.e. the appearance of a crystallographic group of symmetries in the atomic structure created in a crystallization process has always been one of the very basic problems of crystallography. Physicists usually explain it by the fact that, in crystalline matter, local arrangements of atoms of the same kind tend to be identical. Thus the global regular structure is the result of the action of local forces, as captured by the following excerpt from Chapter 30 of~\cite{fey}:\ ``When the atoms of matter are not moving around very much, they get stuck together and arrange themselves in a configuration with as low an energy as possible. If the atoms in a certain place have found a pattern which seems to be of low energy, then the atoms somewhere else will probably make the same arrangement. For these reasons, we have in a solid material a repetitive pattern of atoms.'' Therefore, the space-group symmetry and periodicity of crystal structures can be seen as a result of the requirement of equality of local atomic arrangements throughout the whole crystal.

The mathematical theory that addresses the fundamental problem of relations between the local and global structure of a crystal (namely, how local congruence of atomic arrangements dictates the global structural regularity) was initiated in \cite{dedostga}. The basic idea was to analyze Delone sets in $d$-dimensional Euclidean space $\mathbb{R}^d$ (for precise definitions see below). An ideal crystal structure can be modeled as a regular or multi-regular point system, which is the orbit of a single point or finitely many points under a crystallographic group of isometries in $\mathbb{R}^d$. The local theory proposed first in \cite{dedostga} aims at finding sufficient local conditions for a Delone set $X$ to be a regular or multi-regular system. Thus the local theory seeks to answer the following question:\ which local conditions on Delone sets $X$ of a given type guarantee the emergence of a crystallographic group of symmetries producing $X$ as an orbit set?

A major problem in the local theory concerns the {\em regularity radius\/} $\hat{\rho}_d$ of Delone sets. This is the smallest positive number $\rho$ (depending on $r$ and $R$) with the property that each Delone set $X$ of type $(r,R)$ in $\mathbb{R}^d$ with mutually equivalent $\rho$-clusters is a regular point system. (Here equivalence means congruence under a center preserving isometry.) Thus $\hat{\rho}_d$ is defined by two properties:\ first, each Delone set $X$ of type $(r,R)$ with mutually equivalent point neighborhoods of radius $\hat{\rho}_d$ is a regular system; and second, for any radius $\rho$ smaller than $\hat{\rho}_d$ there exists a  Delone set of type $(r,R)$ with mutually equivalent point neighborhoods of radius $\rho$ which is not a regular system. A priori it is not at all obvious from the definition that this number $\hat{\rho}_d$ exists (but it does!), and how it would depend on $d$, $r$ and $R$. The celebrated ``local criterion for regular systems'' established in~\cite{dedostga} provides the existence of an upper bound for $\hat{\rho_d}$ which depends on $d$, $r$ and $R$ (see also \cite{do15,do16}).

A main goal is to find good upper and lower estimates for the regularity radius $\hat{\rho}_d$ in terms of the radius $R$ of the largest ``empty ball'' that can be inserted into the point system and having no its points inside. (Incidentally, for Delone sets in hyperbolic or spherical spaces there is no upper bound for the regularity radius which is independent of~$r$, so this problem is meaningful only for Euclidean spaces.) For dimensions $1$, $2$, and~$3$, we have the upper bounds $\hat{\rho}_1\leq 2R$, $\hat{\rho}_2\leq 4R$, and $\hat{\rho}_3\leq 10R$, each depending only on $R$ (and~$d$). The proof of the estimate for $d=1$ is straightforward and is presented at the end of Section~\ref{bano}. The estimates for dimensions $2$ and $3$ were obtained by M.~Stogrin (in unpublished work) and, independently, by N.~Dolbilin (this proof just appeared in~\cite{do16}) a long time ago.  The proofs of the bounds for $d=2$ and especially for $d=3$ are quite involved and are based on the local criterion for regular systems of~\cite{dedostga} as well as on the lemma about Delone sets  in 3-space with mutually equivalent $2R$-clusters (this lemma was discovered by Stogrin a long time ago, but was just recently published in \cite{st10}).

For dimensions 1 and $2$, the exact values of the regularity radius are known and are given by $\hat{\rho}_1= 2R$ and $\hat{\rho}_2= 4R$. See the end of Section~\ref{bano} for a proof for $d=1$. For $d=2$, the lower bound $\hat{\rho}_2\geq 4R$ (complementing the aforementioned upper bound $\hat{\rho}_{2}\leq 4R$) is established by an explicit construction of Delone sets of type $(r,R)$ in the plane with mutually equivalent neighborhoods of radius smaller than $4R$, which are not regular systems. More precisely, for every $\varepsilon>0$ there is a Delone set $X$ in $\mathbb{R}^2$ which has mutually equivalent $(4R-\varepsilon)$-clusters but is not a regular system (see~\cite{do15,do17}). A similar construction in higher dimensions shows that $\hat{\rho}_d\geq 4R$ for any $d\geq 2$. However, until quite recently, it was an open question whether or not $\hat{\rho}_d$ was unbounded as a function of $d$. 

If we impose an additional restriction on the symmetries of clusters, or more precisely on existence of only trivial symmetry, then there exists a dimension-independent regularity radius for ``asymmetric'' Delone sets. Namely, if $2R$-clusters are all pairwise congruent and have trivial symmetry, and furthermore, all $4R$-clusters are congruent, then $4R$ is the dimension-independent regularity radius as proved in \cite{dedostga}. On the other hand, if we require all $2R$-clusters to be centrally symmetric, then $2R$ is the dimension-independent regularity radius as it is shown in \cite{do15,do17}, see also \cite{dm16}.

The main result of this paper is Theorem~\ref{thm:2dR}. In this theorem we prove the lower bound 
\[\hat{\rho}_d\geq 2dR\;\;\, (\mbox{for } d\geq 1).\] 
It is remarkable that this new bound grows linearly in $d$. 
When $d=3$ the new lower bound and the upper bound mentioned earlier provide the estimate $6R\leq\hat{\rho}_3\leq 10R$.

The proof of the new lower bound (for $d>2$) is accomplished through an explicit construction of  Delone sets in $\mathbb{R}^d$ with mutually equivalent $(2dR-\varepsilon)$-clusters, which are not regular systems. This construction is inspired by geometric ideas involved in two previously known constructions:\ first, the previously mentioned 2-dimensional construction of Delone sets used to prove that $\hat{\rho}_2\geq 4R$; and second, the 3-dimensional construction of a non-isohedral polyhedral tiling in $\mathbb{R}^3$, due to P.~Engel (\cite{eng2,eng1}), in which the first coronas of any two tiles are equivalent. To honor the influence of Engel's work, we are proposing the term ``Engel set'' for the specific type of Delone set constructed in this paper.

The present paper grew out of the discussions at the sessions of a ``Working Group on Delone Sets'' (consisting of the seven authors) which assembled during the ``Workshop on Soft Packings, Nested Clusters, and Condensed Matter'', held at the American Institute of Mathematics (AIM) in San Jose, California, USA, September 19-23, 2016. These discussions have been inspired by the proof of the two-dimensional lower bound presented in Nikolay~Dolbilin's lecture at the Workshop. This work is the result of collaboration between professional crystallographers (IAB and SVK) and professional mathematicians (MB, ND, NYuE, AG, and ES), which explains its style attempting to combine mathematical rigor with crystallographic intuition. In some places and, in particular, at the end of the paper, we shall interrupt precise mathematical language by illustrative examples from structural crystallography.

We greatly appreciated the opportunity to meet at AIM and are grateful to AIM for its hospitality.

     % Appendices appear after the main body of the text. They are prefixed by
     % a single \appendix declaration, and are then structured just like the
     % body text.
		
\section{Basic Notions}
\label{bano}

Let $r$ and $R$ be positive real numbers with $r<R$. A set $X\subset \mathbb R^d$ is called a \emph{Delone set} of type $(r,R)$, or an \emph{$(r,R)$-system}, if for any ${y}\in\mathbb R^d$ the \emph{open} ball of radius $r$ centered at $y$,
\[ B^o_r({y}):=\{{x}\in\mathbb R^d\colon |{x}-{y}|< r\},\] 
contains at most one point of $X$ and the \emph{closed} ball of radius $R$ centered at $y$,
\[ B_R({y}):=\{{x}\in\mathbb R^d\colon |{x}-{y}|\leqslant R\},\] 
contains at least one point of $X$. Clearly, if $0<r'\leq r<R\leq R'$ then every Delone set of type $(r,R)$ is also a Delone set of type $(r',R')$. In designating a type $(r,R)$ to a Delone set $X$ we usually choose $r$ as large as possible, and $R$ as small as possible. For example, the integer grid $2\mathbb{Z}^2$ (with grid size 2) in the plane is a Delone set of type $(1,\sqrt{2})$. In general, Delone sets can be viewed as universal abstract models of atomic arrangements, where coordinates of points are coordinates of atomic nuclei or gravity centers.

The {\em symmetry group\/} $S(X)$ of a Delone set $X$ in $\mathbb{R}^d$ consists of all isometries of $\mathbb{R}^d$ which map $X$ to itself. This may include both proper and improper isometries.

A \emph{regular system} is a Delone set $X\subset \mathbb R^d$ whose symmetry group $S(X)$ acts transitively on the points of $X$, that is, for any ${x}, {y}\in X$ there exists an isometry $g\in S(X)$ such that $g({x})={y}$. Thus a regular system coincides with the orbit of any of its points under its symmetry group. 

For a point ${x}$ in a Delone set $X$ and for $\rho\geq 0$, we call the subset 
\[C_{x}(\rho):=B_\rho({x})\cap X\] 
of $X$ the {\em cluster of radius\/} $\rho$, or simply the {\em $\rho$-cluster\/},
of $X$ at ${x}$. Note that clusters are ``centered clusters'', in the sense that ${x}$ is distinguished as the ``center'' of $C_{x}(\rho)$. The \emph{cluster group\/} $S_{x}(\rho)$ of a $\rho$-cluster $C_{x}(\rho)$ is defined as the group of all isometries $g$ of $\mathbb{R}^d$ such that $g({x})={x}$ and $g(C_{x}(\rho))=C_{x}(\rho)$. Thus the cluster group $S_{x}(\rho)$ is a subgroup of the full symmetry group of $C_{x}(\rho)$, namely the stabilizer of ${x}$ in the full symmetry group of $C_{x}(\rho)$ (or equivalently, the symmetry group of the ``centered $\rho$-cluster'' at ${x}$).

The $\rho$-clusters $C_{x}(\rho)$ and $C_{{x}'}(\rho)$ at two points ${x},{x}'$ of $X$ are called \emph{equivalent} if there exists an isometry $g$ of $\mathbb{R}^d$ such that $g({x})={x}'$ and $g(C_{x}(\rho))=C_{{x}'}(\rho)$. Note that $g$ is not required to be a symmetry of $X$. Equivalence
is a stronger requirement than mere congruence of the sets $C_{x}(\rho)$ and $C_{{x}'}(\rho)$ under $g$, since the isometry $g$ must also map the center ${x}$ of $C_{x}(\rho)$ to the center $x'$ of $C_{{x}'}(\rho)$.

For any $\rho>0$ the set of $\rho$-clusters of $X$ is partitioned into classes of equivalent $\rho$-clusters. By $N_{X}(\rho)$ we denote the number of equivalence classes of $\rho$-clusters.
We call $N_X$ the {\em cluster counting function\/} of $X$. Clearly, $N_{X}(\rho)=1$ for $\rho<2r$, since then $C_{x}(\rho)=\{{x}\}$ for each ${x}\in X$.

One of the main problems in the local theory of Delone sets is to find small positive numbers $\rho$ with the property that each Delone set $X$ of type $(r,R)$ with mutually equivalent $\rho$-clusters must necessarily be a regular system. The smallest such number $\rho$, denoted $\widehat{\rho_d}=\widehat{\rho_d}(r,R)$, is called the {\em regularity radius\/} and a priori depends on the dimension $d$ and the parameters $r$ and $R$.
\smallskip

It is instructive to look at the 1-dimensional case. A Delone set $X$ on the line is a discrete point set, in which points occur in a natural order and there is an obvious notion of adjacency of points, and the distance between adjacent points is bounded. If $X$ is of type $(r,R)$, then 
\[2r\leq |x-y|\leq 2R\] 
for any two adjacent points $x,y$ of $X$. Given two positive real numbers $a$ and $b$ with $a\leq b$, we call a discrete set $X$ on the line an {\it $ab$-set\/} if the distances between adjacent points of $X$ alternate between $a$ and $b$. An $ab$-set is a Delone set with $r=a/2$ and $R=b/2$. It is easy to see that every $ab$-set is a regular system.

Now suppose that $X$ is a Delone set of type $(r,R)$ on the line with mutually equivalent $\rho$-clusters for some $\rho>0$, and that the $\rho$-clusters of $X$ contain points of $X$ on either side of the center point of the cluster. Then $X$ must necessarily be an $ab$-set, possibly with $a=b$, where $a$ is the smallest non-zero distance of a cluster point from the center, and $b$ is the smallest non-zero distance of a cluster point from the center on the opposite side from the point that determines $a$. It follows that $X$ must be a regular system. 

In any Delone set $X$ of type $(r,R)$, the $2R$-clusters all have the property that points of $X$ lie on either side of the center. Hence, if the $2R$-clusters of $X$ are mutually equivalent, then the previous considerations (with $\rho=2R$) show that $X$ must be a regular system. Hence, for dimension~$1$, the regularity radius satisfies the inequality $\hat{\rho}_{1}\leq 2R$. 

We claim that $\hat{\rho}_{1}=2R$, that is, $\hat{\rho}_1$ cannot be smaller than $2R$. It remains to show that for any positive number $\rho$ with $\rho<2R$ there exists a Delone set $X$ of type $(r,R)$ on the line with mutually equivalent $\rho$-clusters which is not a regular system. So let $0<\rho<2R$. We can construct the point set $X$ by placing points along the line such that adjacent points alternately are at a distance $\rho$ or at a distance strictly between $\rho$ and $2R$ from each other. Then each $\rho$-cluster of $X$ consists of just two points at distance $\rho$, one of which is the center of the cluster. Hence $X$ has mutually equivalent $\rho$-clusters. On the other hand, if any two of the distances which lie between $\rho$ and $2R$ are mutually distinct, then clearly $X$ cannot be a regular system. Thus any such set $X$ is a Delone set of type $(r,R)$, with $r\leq\rho/2$, such that $X$ has mutually equivalent $\rho$-clusters but $X$ is not a regular system. It follows that $\hat{\rho}_{1}=2R$.\smallskip

\section{Construction of Engel Sets}
\label{constrengel}

In this section we present a construction of a certain family of Delone sets in $d$-dimensional space $\mathbb{R}^d$ which we call {\em Delone sets of Engel type\/}, or simply {\em Engel sets\/}. Later these sets are used to establish the new lower bound for the regularity radius. Our construction was inspired by the construction of a non-isohedral polyhedral tiling of 3-space described in Engel~\cite{eng2}. 

The reader wishing to first see examples of Engel sets in dimensions 2 and 3 may skip ahead to Section~\ref{smalldim} and then return to the present section for the discussion of the general $d$-dimensional case. There are (by definition) no Engel sets in dimension~1.
\smallskip

Throughout this section we assume that $d\geqslant 2$. We let $e_1,\dots,e_d$ denote the standard basis in $\mathbb R^d$ and write points of $\mathbb{R}^d$ in row notation in the form $x=(x_{1},\ldots,x_{d})$.

The Engel sets $X$ are determined by four parameters:\ a doubly-infinite integer sequence $A$ with properties (A1), (A2) and (A3) below, and three positive real numbers $a$, $b$ and $\delta$. We write $X=X(A,a,b,\delta)$. Later we place restrictions on the sequence $A$ and the three parameters $a$, $b$ and $\delta$.  

The doubly-infinite integer sequences $A$ are of the form 
$$
A=\big(a_i\big)_{i=-\infty}^{\infty}=(\dots,a_{-2},a_{-1},a_0,a_1,a_2,\dots).
$$
Later the terms $a_i$ of $A$ will determine signed standard basis vectors of $(d-1)$-space $\mathbb{R}^{d-1}$ employed in the construction of Engel sets. The sequence $A$ is required to satisfy the following properties:
\begin{itemize}
\item[(A1)] $a_i\in\{\pm 1,\dots,\pm (d-1)\}$ for each $i\in\mathbb{Z}$;
\item[(A2)] $|a_i|=|a_{i+d-1}|$ for each $i\in\mathbb{Z}$;
\item[(A3)]  $\{|a_1|,\dots,|a_{d-1}|\}=\{1,\dots,d-1\}$.
\end{itemize}
Note that the defining properties (A1), (A2) and (A3) all depend on $d$. The condition (A1) says that the term $a_i$ takes only values from among $\pm 1,\dots,\pm (d-1)$, and condition (A2) means that the related sequence of absolute values, $(|a_i|)_{i=-\infty}^{\infty}$, is ``$(d-1)$-periodic'' under shifts of indices. Thus this  sequence of absolute values is completely determined by the $d-1$ absolute values $|a_1|,\dots,|a_{d-1}|$, which by condition (A3) take precisely the values $1,\dots,d-1$ up to permutation, and thus are mutually different. We call the terms $a_1,\ldots,a_{d-1}$ the {\em initial terms\/} of the sequence $A$.

When $d=2$ the three conditions (A1), (A2) and (A3) simply reduce to the single condition that $a_i=\pm 1$ for each $i$. 

When $d=3$ condition (A1) says that $a_i = \pm 1, \pm2$ for each $i$; condition (A2) requires that $|a_i|=|a_{i+2}|$ for each $i$; and (A3) means that the absolute values of the initial terms $a_1,a_2$ are $1$ and $2$, up to permutation.\smallskip

The Engel sets $X$ in $d$-space $\mathbb{R}^d$ are built layer by layer from translates of a scaled copy of the $(d-1)$-dimensional standard cubic lattice,  
$$
Y=\{2a(m_1,\dots,m_{d-1}): m_{1},\ldots,m_{d-1}\in\mathbb Z\} = 2a\mathbb Z^{d-1},
$$
with grid size $2a$. It is not difficult to see that this grid $Y$ is a regular $\left(a,a\sqrt{d-1}\right)$-system in real $(d-1)$-space $\mathbb{R}^{d-1}$ (here viewed as the linear subspace $\mathbb{R}^{d-1}\times\{0\}$ of $\mathbb{R}^d$). Now, for all sequences $A$ with properties (A1), (A2) and (A3) as above, and for all $a,b,\delta>0$, define the discrete subset $X:=X(A,a,b,\delta)$ of $\mathbb{R}^d$ as a {\em layered set\/} by the following rules:
\begin{itemize}
\item[(E1)] $X$ consists of {\em layers\/} $X_m\subset \{x\in\mathbb  R^d\colon x_d=2bm\}$, at {\em levels\/} $m\in\mathbb Z$, so that
\[X=\bigcup_{m\in \mathbb{Z}} \,X_{m};\]
\item[(E2)] $X_0=Y\times\{0\}=\{(y,0)\in\mathbb{R}^{d-1}\!\times\mathbb R=\mathbb R^d\colon y\in Y\}$;
\item[(E3)] $X_{2i+1}=X_{2i}+2be_d$ for each $i\in \mathbb Z$;
\item[(E4)] $X_{2i}=X_{2i-1}+2be_d+\delta u_i$ for each $i\in\mathbb Z$, where $u_i:={\rm sign}(a_i)e_{|a_i|}$.
\end{itemize}
Here ${\rm sign}(a_i)$ denotes the {\em sign\/} of $a_i$ and is defined as $+1$ if $a_i$ is positive, and $-1$ if $a_i$ is negative. 

See Figure~\ref{pict:engelset} for an example of an Engel set in dimension 3. In the next section we elaborate on Engels sets in dimensions 2 and 3. 

\begin{figure}
\begin{center}
\includegraphics[width=0.15\textwidth]{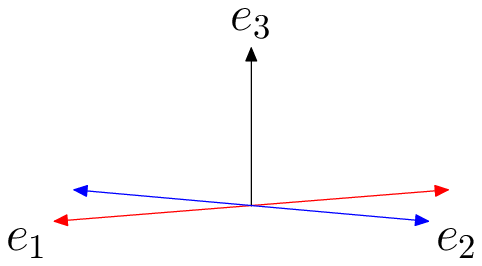}
\hskip 0.5cm
\includegraphics[scale=0.5]{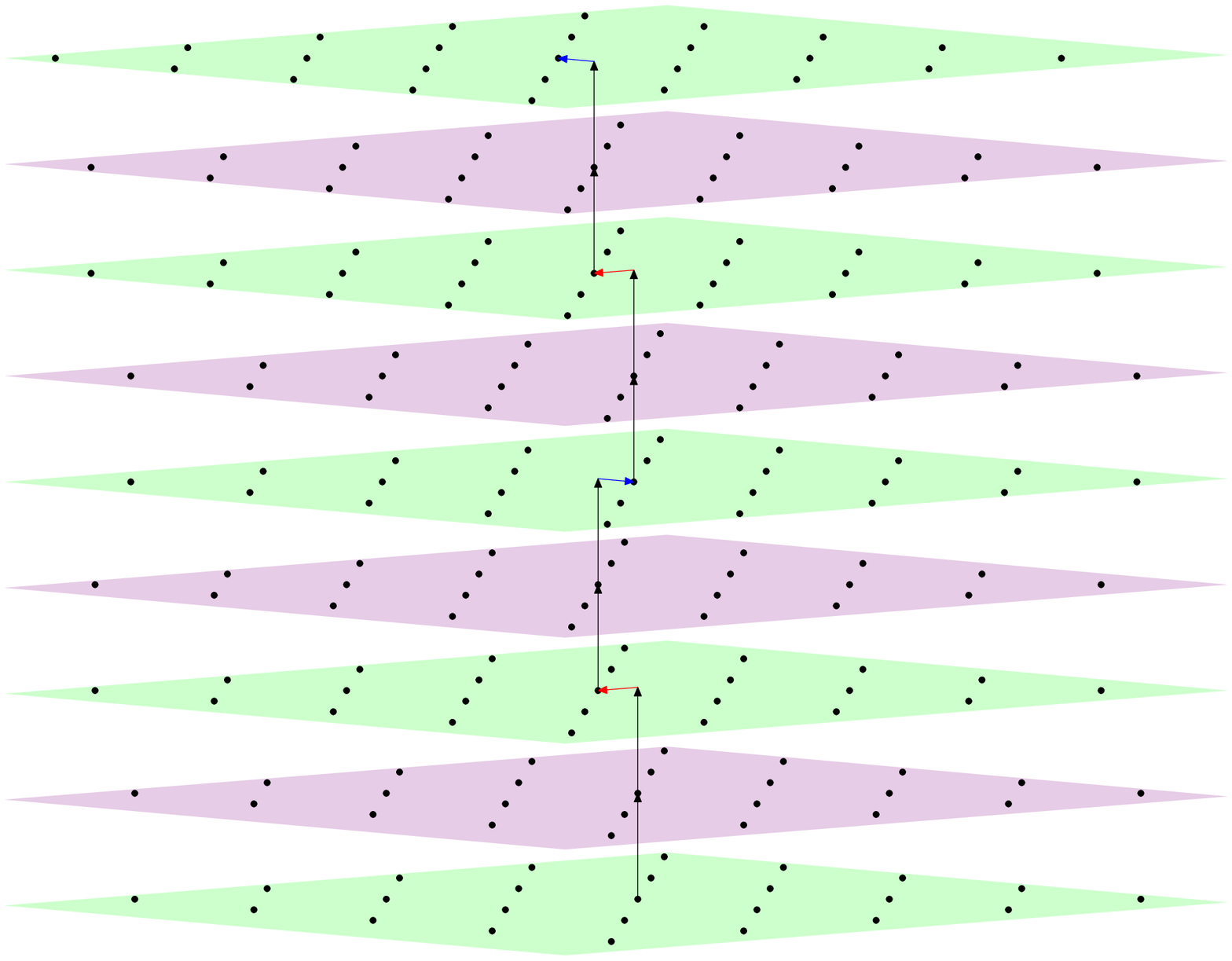}
\hskip 0.5cm
\includegraphics[scale=0.5]{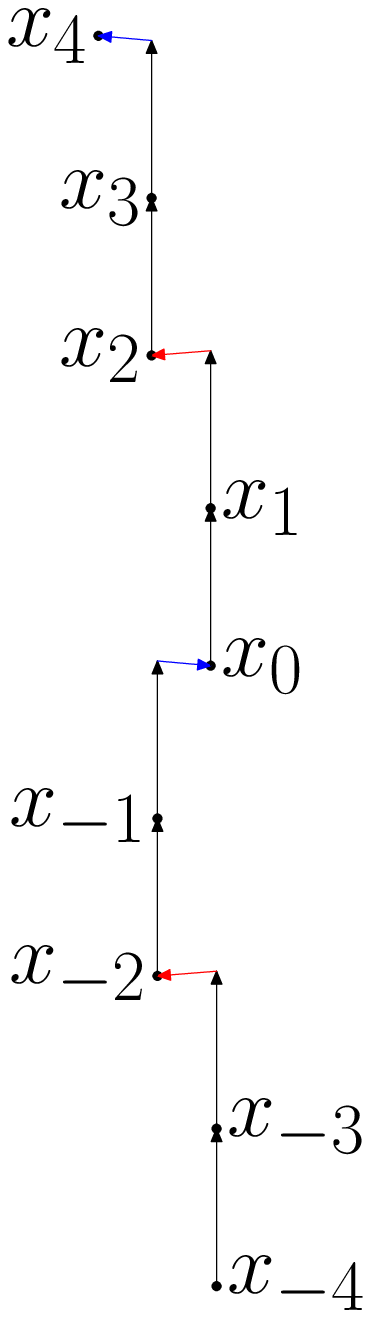}
\end{center}
\caption{Construction of an Engel set in $\mathbb{R}^3$. The small figure on the left shows the standard basis $e_1,e_2,e_3$ of $\mathbb{R}^3$. The vectors $\pm e_1$ are colored in red; the vectors $\pm e_2$ are colored in blue; and the vector $e_3$ is colored in black. The central part of the figure shows a part of the Engel set constructed using a sequence $A$ of the form $A=(\ldots,1,2,1,-2,\ldots)$. Vertical shifts by the vector $be_3$ are shown in black, and horizontal shifts by the vectors $\pm \delta e_1$ and $\pm \delta e_2$ are shown in their respective colors, red or blue. The even and odd layers are colored in green and purple respectively. The right part of the figure shows the sequence $x_{-4},x_{-3},\ldots,x_3,x_4$ associated with the point $x_0$ in the middle layer as defined by the rules (X1), (X2), and (X3) described later in the text.}
\label{pict:engelset}
\end{figure}

Notice that the $(d-1)$-dimensional grid $Y$ occurs as layer $X_0$ in the form $Y\times\{0\}$ at level $0$, and that each layer of $X$ is a translate of $Y$. Any two {\em adjacent\/} layers of $X$ lie in parallel hyperplanes which are at a distance~$2b$. The layer $X_{2i+1}$ at an odd level $2i+1$ is obtained from the layer $X_{2i}$ at the previous level~$2i$ by a ``vertical'' shift by $2be_d$. However, the layer $X_{2i}$ at an even level $2i$ is obtained from the layer $X_{2i-1}$ at the previous level $2i-1$ by a ``vertical'' shift by $2be_d$ followed by a ``horizontal'' shift by $\delta u_i$, where $u_i$ is determined by the $i$-th term $a_i$ of the sequence~$A$ and is given by 
\[u_i:={\rm sign}(a_i)e_{|a_i|}.\] 

Note that the entire set $X$ is invariant under shifts by vectors from the ``horizontal'' grid~$Y$ (that is, the layer $X_0$), and thus under shifts by $2ae_j$ for each $j=1,\dots,d-1$. However, $X$~is not invariant under vertical shifts by $2be_d$. 
\smallskip

Thus Engel sets are discrete layered structures in $d$-space $\mathbb{R}^d$ defined by the conditions (E1)--(E4) above. 

We first show that Engel sets are indeed Delone sets when $b>a$. In fact, in this case we have the following proposition.

\begin{proposition}
\label{delengel}
For all doubly-infinite integer sequences $A$ as above, and for all positive real numbers $a,b,\delta$ with $b>a$, the Engel set $X=X(A,a,b,\delta)$ is an $\left(a,\sqrt{b^2+(d-1)a^2}\right)$-system.
\end{proposition}

\begin{proof}
Let ${y}$ be any point in $\mathbb{R}^d$, and set $r:=a$ and $R:=\sqrt{b^2+(d-1)a^2}$. We need to show that the open ball $B^o_r({y})$ contains at most one point of $X$, and that the closed ball $B_R({y})$ contains at least one point of~$X$.

The first part is simple. In fact, since $b>a$, the smallest distance between two points in $X$ is $2a=2r$. Hence $B^o_r({y})$ contains at most one point of $X$.

Consider the second part. Starting from ${y}$ we can reach a point from $X$ by traveling by a distance of at most $b$ in the direction orthogonal to layers of $X$ to the layer of $X$ closest to $y$, and then traveling from the new point by a distance of at most $a\sqrt{d-1}$ to the closest point of $X$ in this layer (a maximal empty ball for the $2a$-dilation of the $(d-1)$-dimensional cubic lattice $\mathbb{Z}^{d-1}$, that is, for $Y$, has radius $a\sqrt{d-1}$). The Pythagoras Theorem then gives the necessary bound of $\sqrt{b^2+(d-1)a^2}$ for the distance of $y$ to the closest point in $X$. Thus $B_R({y})$ contains at least one point of~$X$. Note that the bound for the distance between $y$ and a point of $X$ is sharp for the point ${y}=(b,a,\ldots,a)$ halfway between the layers $X_0$ and~$X_1$.
\end{proof}

From now on we assume that $b>a$. We also retain the definitions of the parameters $r$ and $R$ from the proof of Proposition~\ref{delengel}, that is, 
\begin{equation}
\label{rR}
r:=a,\;\; R:=\sqrt{b^2+(d-1)a^2}.
\end{equation}
Then Proposition~\ref{delengel} says that $X$ is an $(r,R)$-system. Note that the type $(r,R)$ of $X$ is only determined by the parameters $a$ and $b$, and that 
\[R> r\sqrt{d}\]
since $b>a$. Conversely, if $r$ and $R$ are positive real numbers such that $R> r\sqrt{d}$, then there exist positive real numbers $a$ and $b$ with $b>a$ such that $(r,R)$ is the type of all Delone sets of the form $X=X(A,a,b,\delta)$.

We later impose further conditions on the four parameters $A,a,b,\delta$ of $X$.

\begin{remark}\label{engel_gen}
There is an obvious and natural generalization of the notion of Engel set allowing different distances between adjacent layers (unevenly spaced layers). In this case two distances between adjacent layers, say $2b$ and $2b'$, should alternate giving rise of a new set depending on parameters $A, a, b, b', \delta$. We will also refer to these more general sets as Engel sets.
\end{remark}

\section{Engel Sets in Dimensions 2 and 3}
\label{smalldim}

In this section we discuss the structure of Engel sets in small dimensions, 2 and 3. As before there are four parameters involved, namely a doubly-infinite integer sequence $A=(a_i)_{i=-\infty}^{\infty}$ and three positive real parameters $a$, $b$ and $\delta$. Due to the nature of $A$, the formalism for the construction of corresponding Engel sets $X(A,a,b,\delta)$ simplifies considerably in dimensions~$2$ and $3$. 

\subsection{Planar Engel sets}
\label{planar}

When $d=2$ the defining properties (A1), (A2) and (A3) for the doubly-infinite integer sequence $A=(a_i)_{i=-\infty}^{\infty}$ reduce to the single condition that 
\[ a_{i} = \pm 1\;\;\, (i\in\mathbb{Z}).\]
Thus $A$ can be any doubly-infinite sequence of $1$'s and $-1$'s. The standard basis vectors of the plane $\mathbb{R}^2$ are given by $e_{1}=(1,0)$ and $e_{2}=(0,1)$, and the points in $\mathbb{R}^2$ are denoted as usual by $(x,y)$.

The Engel sets $X=X(A,a,b,\delta)$ in $\mathbb{R}^2$ consist of ``layers'' of horizontal 1-dimensional grids of grid size $2a$ (in a sense, copies of $Y=2a\mathbb{Z}$). More precisely, 
\[X=\bigcup_{m\in \mathbb{Z}} \,X_{m},\]
where each layer $X_m$ is a translate of the layer
\[X_{0}:=Y\times \{0\}=2a\mathbb{Z}\times\{0\}\] 
lying in the horizontal line $y=2mb$ in $\mathbb{R}^2$. Note that any two adjacent horizontal lines are at distance $2b$ from each other. Each layer of $X$ at an odd level $2i+1$ ($i\in\mathbb{Z}$) is simply obtained from the layer at the previous even level $2i$ by a vertical shift by the vector $2be_{2}=(0,2b)$, that is, 
\[ X_{2i+1}=X_{2i}+ (0,2b).\]
On the other hand, since $|a_i|=1$ for all $i\in\mathbb{Z}$, each layer at an even level $2i$ ($i\in\mathbb{Z}$) is obtained from the layer at the previous odd level $2i-1$ by a shift by the vector 
\[2be_2+\delta u_i = 2be_{2}+ {\rm sign}(a_i)\delta e_{1} = ({\rm sign}(a_i)\delta,2b),\]
that is, 
\[X_{2i}=X_{2i-1}+({\rm sign}(a_i)\delta,2b).\]
In other words, while going from $X_{2i}$ to $X_{2i+1}$ is straightforward and just involves a vertical shift by $(0,2b)$, the shift involved in passing from $X_{2i-1}$ to $X_{2i}$ is determined by $\delta$ and the sign of the $i^{th}$ term $a_i$ of the sequence $A$ (if $a_i>0$ this shift is by $(\delta,2b)$, and if $a_{i}<0$ the shift is by $(-\delta,2b)$).

From Proposition~\ref{delengel} we know that, when $d=2$ and $b>a$, the parameters $r$ and $R$ giving the type $(r,R)$ of $X(A,a,b,\delta)$ as a Delone set are given by
\[ r=a, \;\;R=\sqrt{a^2+b^2}.\]

Note that the special features of Engel sets $X(A,a,b,\delta)$ in any dimension are most apparent and amplified when the parameters $a,b,\delta$ are chosen in such a way that $b$ is (much) larger than $a$, and $\delta$ is (much) smaller than $a$. However, a priori we are not making these assumptions on $a$, $b$ and~$\delta$.

As an example we consider the planar Engel set $X(A,a,b,\delta)$ obtained for the following parameter values. Suppose $A$ is the sequence 
\[(\ldots,1,1,-1,1,1,-1,1,1,-1,\ldots),\] 
with $a_{1}=a_{2}=1$ and $a_{3}=-1$, which repeats $(1,1,-1)$-strings indefinitely, in either direction. Suppose further that $a=5$, $b=12$, and $\delta=1$, so that $r=a=5$ and $R=\sqrt{a^{2}+b^2}=13$. (For simplicity we have chosen integer values for $a$ and $b$ such that $R$ is also an integer.) For the corresponding Engel set $X(A,a,b,\delta)$ we recorded the thirteen layers $X_{-6},X_{-5},\ldots,X_6$ in Table~\ref{planarengel}; see the Figure~\ref{pict:2dim-engel}. The second column in the table describes the collection of points in a layer $X_i$. For example, layer $X_2$ is given by 
%$$
%X_{2}=(1+10\mathbb{Z})\times \{48\}=\{(1+10j,48)\mid j\in\mathbb{Z}\}
%= (1,48)+\,Y\times\{0\} 
%= \{\ldots,(-19,48),(-9,48),(1,48),(11,48),(21,48),\ldots\},
%$$
\[\begin{array}{rl}
X_{2}&=(1+10\mathbb{Z})\!\times\! \{48\}\!\!\!=\{(1+10j,48)\mid j\in\mathbb{Z}\}\\[.04in]
&= (1,48)+\,Y\!\times\!\{0\} \\[.04in]
&= \{\ldots,(-19,48),(-9,48),(1,48),(11,48),(21,48),\ldots\},
\end{array}\]
and is contained in the horizontal line $y=48=4b$ of $\mathbb{R}^2$.

\begin{table}
\centering
\begin{tabular}{|l|l|l|}
\hline
Layer $X_m$& Coordinates& Relationship to Layer $X_{m-1}$\\
\hline
$X_{-6}$&$(-1+10\mathbb{Z})\times \{-144\}$&$X_{-7} + (-1,24)$\\[.02in]
$X_{-5}$&$(-1+10\mathbb{Z})\times \{-120\}$&$X_{-6} + (0,24)$\\[.02in]
$X_{-4}$&$10\mathbb{Z}\times \{-96\}$&$X_{-5} +(1,24)$ \\[.02in]
$X_{-3}$&$10\mathbb{Z}\times \{-72\}$&$X_{-4} + (0,24)$ \\[.02in]
$X_{-2}$&$(1+10\mathbb{Z})\times \{-48\}$&$X_{-3} +(1,24)$ \\[.02in]
$X_{-1}$&$(1+10\mathbb{Z})\times \{-24\}$&$X_{-2} + (0,24)$ \\[.02in]
$X_{0}$&$10\mathbb{Z}\times \{0\}$&$X_{-1} + (-1,24)$ \\[.02in]
$X_{1}$&$10\mathbb{Z}\times \{24\}$&$X_{0}+(0,24)$ \\[.02in]
$X_{2}$&$(1+10\mathbb{Z})\times \{48\}$&$X_{1}+(1,24)$ \\[.02in]
$X_{3}$&$(1+10\mathbb{Z})\times \{72\}$&$X_{2}+(0,24)$ \\[.02in]
$X_{4}$&$(2+10\mathbb{Z})\times \{96\}$&$X_{3}+(1,24)$ \\[.02in]
$X_{5}$&$(2+10\mathbb{Z})\times \{120\}$&$X_{4}+(0,24)$ \\[.02in]
$X_{6}$&$(1+10\mathbb{Z})\times \{144\}$&$X_{5}+(-1,24)$\\[.02in]
\hline
\end{tabular}
\caption{Coordinates for the layers $X_{-6},X_{-5},\ldots,X_6$ of the planar Engel set $X(A,a,b,\delta)$, with parameters $A=(\ldots,1,1,-1,1,1,-1,1,1,-1,\ldots)$, $a=5$, $b=12$, and $\delta=1$. This Engel set is a Delone set of type $(5,13)$ with mutually equivalent clusters of radius $48$.}
\label{planarengel}
\end{table}

\begin{figure}
\begin{center}
\includegraphics[width=0.8\textwidth]{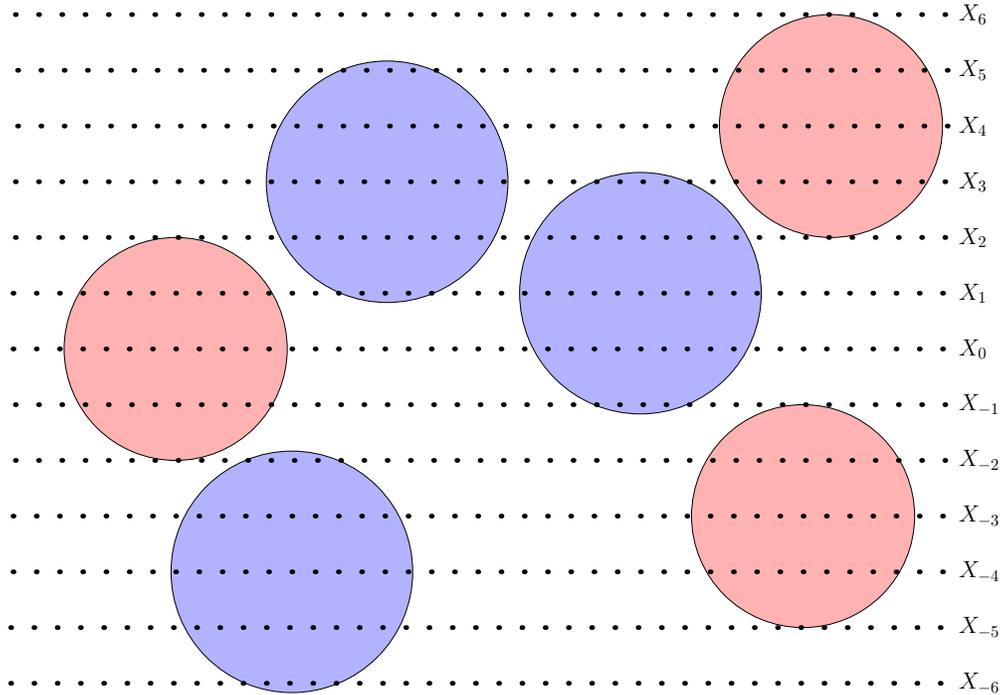}
\end{center}
\caption{An example of a planar Engel set with highlighted $4R$-clusters (blue) and $(4R-\varepsilon)$-clusters (red).}
\label{pict:2dim-engel}
\end{figure}

Later in this paper we establish a general result from which it will follow that this particular Engel set $X$ has the property that its clusters of radius~$48=4R-4$ are mutually equivalent. In fact, if we set $\varepsilon:=4$ then the parameters $a$ and $b$ satisfy the conditions~(\ref{crucial}) of the proof of Theorem~\ref{thm:2dR}, namely $a<b$ and $a^{2}<\varepsilon b/2$, so that Theorem~\ref{thm:2dR} shows that $X$ has mutually equivalent clusters of radius $4R-\varepsilon=48$; see the red clusters in Figure~\ref{pict:2dim-engel}. On the other hand, by the particular choice of $A$, this Engel set $X$ cannot be a regular system by Corollary~\ref{cor:regularity}. (In fact, Corollary~\ref{cor:regularity} says that, up to isometry, regular systems can only arise if the terms $a_i$ of $A$ are all equal to $1$, are all equal to $-1$, or alternate between $1$ and $-1$.) Then in turn, being a non-regular Engel set, $X$ cannot have mutually equivalent clusters of radius $4R=52$ by Theorem~\ref{enreg}; see the blue clusters in Figure~\ref{pict:2dim-engel}. The two blue clusters on the left are equivalent, while the rightmost blue cluster is not equivalent to them.  In the two blue clusters on the left, a vertical ``column'' of points takes two small shifts to the right, while in the rightmost blue cluster, a vertical ``column'' takes one shift to the right and one shift to the left. Hence $N_X(48)=1$, but $N_X(52)>1$. 

For the current example, the radius can be slightly increased from $48$ to 48.15 and still have mutually equivalent clusters; that is, even $N_X(48.15)=1$. This can be seen by applying Theorem~\ref{thm:2dR} with a slightly smaller $\varepsilon$ than $4$. 

Thus, as the above Engel set shows, mutual equivalence of clusters of radius 48 (or even $48.15$) is not enough to imply regularity of a Delone set of type (5,13). 

In our example, the cluster groups of clusters of radius greater than or equal to $2R=26$ are trivial. This is consistent with Theorem~\ref{Gsym} (for $k=1$), which we establish in the next section.

\subsection{Engel sets in 3 Dimensions}
\label{3space}

The standard basis vectors of 3-space $\mathbb{R}^3$ are given by $e_{1}=(1,0,0)$, $e_{2}=(0,1,0)$, and $e_{3}=(0,0,1)$, and points in $\mathbb{R}^3$ are denoted by $(x,y,z)$. The defining properties (A1), (A2) and (A3) for the doubly-infinite integer sequence $A=(a_i)_{i=-\infty}^{\infty}$ now reduce to the following three conditions:
$$
a_i=\pm 1,\pm 2\;\, (i\in\mathbb{Z}),\;\;\, (|a_1|,|a_2|) = (1,2),(2,1), |a_{i+2}|=|a_{i}|\;\,(i\in\mathbb{Z}).
$$
The three remaining parameters $a,b,\delta$ are positive real numbers. 

The Engel sets $X=X(A,a,b,\delta)$ in $\mathbb{R}^3$ consist of ``layers'' of horizontal 2-dimensional grids of grid size $2a$ (that is, copies of $Y=2a\mathbb{Z}^2$). Now 
\[X=\bigcup_{m\in \mathbb{Z}} \,X_{m},\]
where each layer $X_m$ is a translate of the layer
\[X_{0}:=Y\times \{0\}=2a\mathbb{Z}^2\times\{0\}\] 
lying in the horizontal plane $z=2mb$ in $\mathbb{R}^3$. Here any two adjacent horizontal planes lie at distance $2b$ from each other. Each layer of $X$ at an odd level $2i+1$ ($i\in\mathbb{Z}$) is just obtained from the layer at the previous even level $2i$ by a vertical shift by the vector $2be_{3}=(0,0,2b)$, that is, 
\[ X_{2i+1}=X_{2i}+ (0,0,2b).\]
However, each layer at an even level $2i$ ($i\in\mathbb{Z}$) is derived from the layer at the previous odd level $2i-1$ by a shift by the vector 
\[2be_3+\delta u_i =2be_3 + {\rm sign}(a_i)\delta e_{|a_i|},\]
that is, 
$$
X_{2i}=X_{2i-1}+({\rm sign}(a_i)\delta,0,2b)\;\mbox{ or }\;
X_{2i}=X_{2i-1}+(0,{\rm sign}(a_i)\delta,2b)
$$
according as $|a_i|=1$ or $2$.
While going from $X_{2i}$ to $X_{2i+1}$ is straightforward as in the 2-dimensional and just involves a vertical shift by $(0,0,2b)$, the shift involved in passing from $X_{2i-1}$ to $X_{2i}$ now is determined by $\delta$ and the $i^{th}$ term $a_i$ itself, not just its sign. 

In the 3-dimensional case Proposition~\ref{delengel} is telling us that, when $b>a$, the parameters $r$ and $R$ of $X(A,a,b,\delta)$ are given by
\[ r=a, \;\;R=\sqrt{2a^2+b^2}.\]

Let us look at the example of the Engel set $X(A,a,b,\delta)$ in 3-space obtained for the following parameter values. The infinite sequence $A$ is given by 
\[A=(\ldots,1,2,-1,2,1,2,-1,2,1,2,-1,2,\ldots),\] 
where $a_{1}=1$, $a_{2}=2$, $a_{3}=-1$, and $a_{4}=2$, and $(1,2,-1,2)$-strings are repeated indefinitely in either direction. The remaining parameters are $a=4$, $b=7$, and $\delta=1$, so that $r=a=4$ and $R=\sqrt{2a^{2}+b^2}=9$. Thus the corresponding Engel set $X(A,a,b,\delta)$ is a Delone set of type~$(4,9)$. The thirteen layers $X_{-6},X_{-5},\ldots,X_6$ of $X(A,a,b,\delta)$ are listed in Table~\ref{spaceengel}. The second column again describes the collection of points in a layer $X_i$ using notation similar to the planar case. For example, layer $X_4$ is given by 
%$$
%X_{4}=(1+8\mathbb{Z})\!\times\! (1+8\mathbb{Z})\times \{56\} 
%= \{(1+8j,1+8k,56)\mid j,k\in\mathbb{Z}\}= (1,1,56) +8\mathbb{Z}\times 8\mathbb{Z}\times \{56\} 
%=(1,1,56) +Y\!\times\! \{56\},
%$$
\[\begin{array}{lll}
X_{4}&=&(1+8\mathbb{Z})\!\times\! (1+8\mathbb{Z})\!\times\! \{56\} = \{(1+8j,1+8k,56)\mid j,k\in\mathbb{Z}\}\\[.04in]
&=& (1,1,56) +8\mathbb{Z}\!\times\! 8\mathbb{Z}\!\times\! \{56\} \\[.04in]
&=&(1,1,56) +Y\!\times\! \{56\},
\end{array} \]
and is contained in the line horizontal plane $z=56=8b$ of $\mathbb{R}^3$.

\begin{table}
\centering
\begin{tabular}{|l|l|l|}
\hline
Layer $X_m$& Coordinates& Relationship to Layer $X_{m-1}$\\
\hline
$X_{-6}$&$(1+8\mathbb{Z})\times (-2+8\mathbb{Z})\times \{-84\}$&$X_{-7} + (1,0,14)$\\[.02in]
$X_{-5}$&$(1+8\mathbb{Z})\times (-2+8\mathbb{Z})\times \{-70\}$&$X_{-6} + (0,0,14)$\\[.02in]
$X_{-4}$&$(1+8\mathbb{Z})\times (-1+8\mathbb{Z})\times \{-56\}$&$X_{-5} +(0,1,14)$ \\[.02in]
$X_{-3}$&$(1+8\mathbb{Z})\times (-1+8\mathbb{Z})\times \{-42\}$&$X_{-4} + (0,0,14)$ \\[.02in]
$X_{-2}$&$8\mathbb{Z}\times (-1+8\mathbb{Z})\times \{-28\}$&$X_{-3} +(-1,0,14)$ \\[.02in]
$X_{-1}$&$8\mathbb{Z}\times (-1+8\mathbb{Z})\times \{-14\}$&$X_{-2} + (0,0,14)$ \\[.02in]
$X_{0}$&$8\mathbb{Z}\times 8\mathbb{Z}\times \{0\}$&$X_{-1} + (0,1,14)$ \\[.02in]
$X_{1}$&$8\mathbb{Z}\times 8\mathbb{Z}\times \{14\}$&$X_{0}+(0,0,14)$ \\[.02in]
$X_{2}$&$(1+8\mathbb{Z})\times 8\mathbb{Z}\times \{28\}$&$X_{1}+(1,0,14)$ \\[.02in]
$X_{3}$&$(1+8\mathbb{Z})\times 8\mathbb{Z}\times \{42\}$&$X_{2}+(0,0,14)$ \\[.02in]
$X_{4}$&$(1+8\mathbb{Z})\times (1+8\mathbb{Z})\times \{56\}$&$X_{3}+(0,1,14)$ \\[.02in]
$X_{5}$&$(1+8\mathbb{Z})\times (1+8\mathbb{Z})\times \{70\}$&$X_{4}+(0,0,14)$ \\[.02in]
$X_{6}$&$8\mathbb{Z}\times (1+8\mathbb{Z})\times \{84\}$&$X_{5}+(-1,0,14)$\\[.02in]
\hline
\end{tabular}
\caption{Coordinates for the layers $X_{-6},X_{-5},\ldots,X_6$ of the Engel set $X(A,a,b,\delta)$ in 3-space, with parameters $A=(\ldots,1,2,-1,2,1,2,-1,2,1,2,-1,2,\ldots)$, $a=4$, $b=7$, and $\delta=1$. This Engel set is a Delone set of type $(4,9)$ with mutually equivalent clusters of radius $48$.}
\label{spaceengel} 
\end{table}

Similar remarks as for our planar example also apply to our 3-dimensional example. Appealing as before to 
Theorem~\ref{thm:2dR}, now with $\varepsilon = 14$, we find that the Engel set $X$  in 3-space has the property that its clusters of radius~$40=6R-14$ are mutually equivalent; in fact, the relevant conditions $a<b$ and $a^2<\frac{\varepsilon b}{6}$ are satisfied in this case. On the other hand, by the particular choice of $A$, this Engel set $X$ cannot be a regular system by Corollary~\ref{cor:regularity}. (Note that when $a_{1}=1$ and $a_{2}=2$ only the two sequences $(\ldots,1,2,1,2,1,2,1,2,1,2,1,2,\ldots)$ and $(\ldots,1,2,-1,-2,1,2,-1,-2,1,2,-1,-2,\ldots)$ produce Engel sets which are regular systems.) It follows that, as a non-regular Engel set, $X$ cannot have mutually equivalent clusters of radius $6R=54$ by Theorem~\ref{enreg}. Hence $N_X(40)=1$, but $N_X(54)>1$. 

Figure~\ref{pict:3dim-engel} shows the above Engel set and two of its $6R$-clusters (blue) and $(6R-\varepsilon)$-clusters (red). It is not clear from the picture that the two blue clusters are not equivalent. However, as we will show later in Lemma~\ref{2Lclose}, the only way for these clusters to be equivalent under an isometry is to map corresponding layers to each other. Hence $X_{-3}$ must be mapped to $X_{-1}$, $X_{-2}$ to $X_0$, and so on. In this case the horizontal parts of the shifts between consecutive layers of one cluster must be mapped onto the horizontal parts of the shifts between layers of the other cluster. Thus the vector $e_2$ (horizontal part of the shift from $X_{-1}$ to $X_0$ of the right blue cluster) must be mapped to $-e_1$ (horizontal part of the shift from $X_{-3}$ to $X_{-2}$ of the left blue cluster). On
the other hand, the vector $e_2$ (horizontal part of the shift from $X_{3}$ to $X_4$ of the right blue cluster) must be mapped to $e_1$ (horizontal part of the shift from $X_{1}$ to $X_{2}$ of the left blue cluster). However, these two conditions cannot be simultaneously satisfied by a single isometry. Thus the two clusters cannot be equivalent. More details, as well as the discussion for the general case, can be found in the proof of Lemma \ref{2Lclose}.

\begin{figure}
\begin{center}
\includegraphics[width=0.8\textwidth]{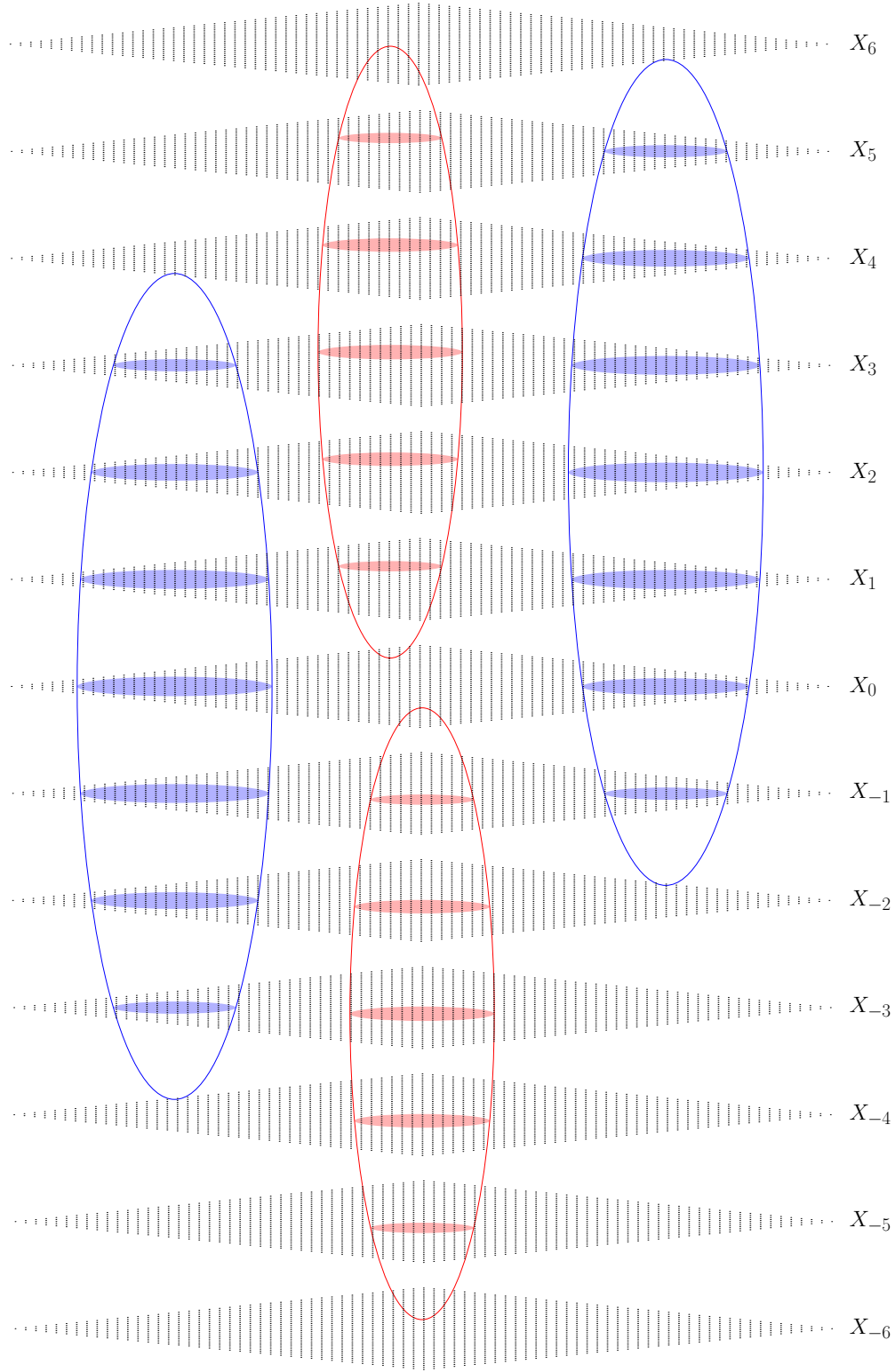}
\end{center}
\caption{An example of an Engel set in 3-space with highlighted $6R$-clusters (blue) and $(6R-\varepsilon)$-clusters (red).}
\label{pict:3dim-engel}
\end{figure}

For this 3-dimensional example the radius can be increased to 40.28 and still have mutually equivalent clusters; that is, $N_X(40.28)=1$; here we can use Theorem~\ref{thm:2dR} with an $\varepsilon$ slightly smaller than $14$. 

Thus, our Engel set in 3-space shows that mutual equivalence of clusters of radius 40 (or even $40.28$) in a Delone set of type (4,9) is not enough to imply regularity. 

In our example, the cluster groups of clusters of radius $2R=18$ have order $2$ and are generated by a reflection (in the plane $x=0$ of $\mathbb{R}^3$), while those for clusters of radius greater than or equal to $4R=36$ are trivial. This is consistent with Theorem~\ref{Gsym} proved in the next section. 

\section{Engel Sets and Regularity}
\label{main}

Throughout this section let $d\geq 2$ (unless said otherwise) and 
\[X:=X(A,a,b,\delta)=\bigcup_{m\in \mathbb{Z}} \,X_{m}\]
be the Engel set in $d$-space $\mathbb{R}^d$ defined as in Section \ref{constrengel} by the doubly-infinite integer sequence $A$ and the three positive real parameters $a,b,\delta$. Let again $b>a$. Then we know from Proposition~\ref{delengel} that $X$ is a Delone set of type $(r,R)$, with $r=a$ and $R=\sqrt{b^2+(d-1)a^2}$. We define 
\[\mathcal{E}:=\{e_1,\dots,e_{d-1}\},\] 
which is the set of the first $d-1$ vectors from the standard basis vectors $e_1,\ldots,e_d$ of $\mathbb{R}^d$.

In the following we often need to consider sets (or rather strings) of successively adjacent layers of $X$, as well as corresponding sets of basis vectors from $\mathcal{E}$ involved in (E4) in the construction of layers at even levels. In particular, for $p\in \mathbb{Z}$ and $1\leqslant k\leqslant d-1$ let $\mathcal{E}_{p,k}$ denote the subset of $\mathcal{E}$ consisting of the basis vectors that appear in the construction of the layers at even levels among the set of $2k+1$ layers
\[X(p,k):= \{X_{p-k},X_{p-k+1},\dots,X_p,\dots,X_{p+k}\}\]
of $X$. Note that in the construction of the layers from $X(p,k)$ we used the rules (E3) and (E4) in total $2k$ times, namely each rule exactly $k$ times. Then, 
\[\mathcal{E}_{p,k}=\{e_{|a_{t_p}|},\dots,e_{|a_{t_{p}+k-1}|}\}\] 
for some integer $t_p$. In particular, when $k=d-1$ the defining properties (A2) and (A3) for the sequence $A$ show that $\mathcal{E}_{p,d-1}=\mathcal{E}$ for each $p\in\mathbb{Z}$.
\smallskip

We first establish a technical lemma describing a sufficient condition for the equivalence of any two $\rho$-clusters of an Engel set $X$.

\begin{lemma}
\label{2dm1Lemma}
Let $\rho>0$. Suppose that, for every $p\in\mathbb{Z}$ and every point $x\in X_p$, the $\rho$-cluster $C_\rho(x)$ of $X$ at $x$ contains only points from layers $X_l$ with $|l-p|\leqslant d-1$. Then the $\rho$-clusters at any two points of $X$ are equivalent; that is, $N_{X}(\rho)=1$.
\end{lemma}

\begin{proof}
We will prove that for any $p,p'\in\mathbb Z$ and any two points $x\in X_p$, $x'\in X_{p'}$ the $\rho$-clusters $C_\rho(x)$ and $C_\rho(x')$ are equivalent. First note that by our assumptions on $X$, the $\rho$-clusters at points in $X_p$ and $X_{p'}$ only contain points from layers in the sets of layers $X(p,d-1)$ or $X(p',d-1)$, respectively. The case $p=p'$ is simple. In fact, since $X$ is invariant under shifts by vectors from $Y$, any two $\rho$-clusters at points in $X_p$ are even equivalent under the translation subgroup of the symmetry group of $X$. Thus to prove equivalence in the general case it is sufficient to show the existence of an isometry of $\mathbb R^d$ that maps the union of the layers in $X(p,d-1)$ to the union of the layers in $X(p',d-1)$.

The $2d-1$ layers in $X(p,d-1)$ and $X(p',d-1)$ are connected by formulas (E3) and (E4). Using the above notation we can say that the layers  in $X(p,d-1)$ involve the $d-1$ shifts by the vectors 
\[2be_d+\delta u_t,\;\; t\in\{t_{p},t_{p+1},\ldots,t_{p+d-2}\}\] 
(recall here that $\{e_{|a_{t_p}|},\dots,e_{|a_{t_{p}+d-2}|}\}=\mathcal{E}_{p,d-1}=\mathcal{E}$), as well as $d-1$ shifts by the vector $2be_d$. The shifts of these two types alternate. For $X(p',d-1)$ the same is true with $p$ replaced by~$p'$.

If $p$ and $p'$ have the same parity, then the orders in which the shift types occur are the same. Hence for an isometry that maps the union of the layers in $X(p,d-1)$ to the union of the layers in $X(p',d-1)$, we may choose a shift that maps $X_p$ to $X_{p'}$, followed by the isometry defined on the standard basis $e_1,\ldots,e_d$ of $\mathbb{R}^d$ via the assignments 
\[u_{t_p+s}\to u_{t_{p'}+s}\;\; (s=0,\dots,d-2)\] 
and $e_d\to e_d$.

If $p$ and $p'$ have different parity, then we first apply the reflection in the hyperplane $x_{d}=b$ of $\mathbb{R}^d$ whose linear part maps $e_d$ to $-e_d$ and fixes $e_1,\ldots,e_{d-1}$. This reflection sends $X_0$ to $X_1$ and maps each layer $X_m$ of $X$ to a layer $X'_{m'}$ of a new layered set $X'=X'(A',a,b,\delta)$ built in a similar way as $X$ (with the same parameters $a,b,\delta$ but a new integer sequence, $A'$).

Then each $X'_{m'}$ is again a translate of $Y$, and $X_0'=X_0$. We can now proceed as before, for two reasons:\   first, the parity of each layer is altered by the reflection in the hyperplane $x_{d}=b$; and second, the proof of the previous case (when $p$ and $p'$ had the same parity) works equally well if the two collections of layers are taken from different Engel sets with the same parameters $a,b,\delta$ (but possibly different sequences $A$ and $A'$) rather than from the same Engel set. This completes the proof.
\end{proof}

We note the following important consequence of the previous lemma. Here we describe sufficient conditions on the parameters $a$ and $b$ which allow us to conclude that any two clusters of $X$ of radius $2dR-\varepsilon$ are equivalent. Recall that the parameter $R$ is given by $R=\sqrt{b^2+(d-1)a^2}$.

\begin{theorem}\label{thm:onecluster} 
Let $d\geq 2$ and $\varepsilon>0$. For all sequences $A$ as above, and all $a,b,\delta>0$ with $a<b$ and $2dR-\varepsilon<2db$, the cluster counting function of the Engel set $X:=X(A,a,b,\delta)$ satisfies $N_{X}(2dR-\varepsilon)=1$. In particular, this property holds when $a^2\leqslant\frac{\varepsilon b}{d(d-1)}$.
\end{theorem}

\begin{proof}
First note that $X$ depends on $\varepsilon$, by the choice of the parameters $a$ and $b$. The first statement of the theorem follows directly from Lemma~\ref{2dm1Lemma}, since the term $2db$ occurring in the condition on $R$ is the distance between two hyperplanes that contain  layers of $X$ that are $2d$ steps apart. For the second statement note that
\[\begin{array}{lcl}
2dR-\varepsilon  &=& 2d\sqrt{b^2+(d-1)a^2}-\varepsilon \\
&=& 2db\sqrt{1+(d-1)\left(\frac{a}{b}\right)^2}-\varepsilon \\
&<& 2db\left(1+\frac{(d-1)a^2}{2b^2}\right)-\varepsilon\\
&=& 2db+\left(\frac{d(d-1)a^2}{b}-\varepsilon\right),\\
\end{array}\]
where we used the trivial inequality $\sqrt{1+x}<1+\frac{x}{2}$ for positive real numbers $x$.

Hence, if $a^2\leqslant\frac{\varepsilon b}{d(d-1)}$ then $2dR-\varepsilon<2db$ and so the condition on $R$ for the first statement is satisfied.
\end{proof}

Next we investigate the cluster groups of clusters in 
\[X=X(A,a,b,\delta)=\bigcup_{p\in \mathbb{Z}} X_p.\] We begin with two technical lemmas. 

For each $p\in \mathbb{Z}$ and each $x\in X_p$ define the sequence of points $(x_j)_{j=-\infty}^{\infty}$ by the following rules (see the right part of Figure \ref{pict:engelset} for an illustration of some of these  points):
\begin{itemize}
\item[(X1)]$x_0=x$;
\item[(X2)]$x_{j+1}=x_j+2be_d\in X_{p+j+1}$, if $p+j$ is even (and $p+j=2i$); 
\item[(X3)]$x_{j+1}=x_j+2be_d+\delta u_i\in X_{p+j+1}$, if $p+j$ is odd and $p+j=2i-1$.
\end{itemize}
Then this sequence of points from $X$ satisfies the following properties.

\begin{lemma}
\label{Lclose}
Suppose $\delta<a$. Let $p\in \mathbb{Z}$ and $x\in X_p$, and let $(x_j)_{j=-\infty}^{\infty}$ be the sequence of points from $X$ associated with $x$ as above. Then, for each $j\in \mathbb Z$, the point $x_j$ is the unique point of the layer $X_{p+j}$ closest to both point $x_{j+1}$ and $x_{j-1}$, with the distances given by
$$
(|x_j-x_{j-1}|,|x_{j+1}-x_j|)=\begin{cases}\left(2b,\sqrt{\delta^2+4b^2}\right),\text{ if }p+j\text{ is odd};\\
\left(\sqrt{\delta^2+4b^2},2b\right),\text{ if }p+j\text{ is even}.
\end{cases}
$$
\end{lemma}

\begin{proof}
We demonstrate the case when $p+j$ is odd, with $p+j=2i-1$; the case when $p+j$ is even is similar.  Then the line segment from $x_{j-1}$ to $x_j$ is the perpendicular from $x_{j-1}$ to the layer $X_{p+j}$; hence $x_j$ is the closest point of $X_{p+j}$ to $x_{j-1}$, and $|x_j-x_{j-1}|=2b$, as claimed. Further, recall that any point $x'\in X_{p+j}$ has the form
$$
x_{j+1}-2be_{d}-\delta u_i+(2a m_1,\dots,2a m_{d-1},0),
$$
where $m_1,\dots,m_{d-1}\in \mathbb Z$ and $u_i={\rm sign}(a_i)e_{|a_i|}$. Set $s:=|a_i|$. Then $x'$ satisfies
%$$
%|x_{j+1}-x'|^2  =  |(2am_1,\dots,2am_s\pm \delta,\dots,2a m_{d-1},-2b)|^2
%= 4a^2\left(m_1^2+\dots+\left(m_s\pm\frac{\delta}{2a}\right)^2+\dots+m_{d-1}^2\right)+4b^2
%$$
\[\begin{array}{lcl}
|x_{j+1}-x'|^2 & = & |(2am_1,\dots,2am_s\pm \delta,\dots,2a m_{d-1},-2b)|^2\\
&=& 4a^2\left(m_1^2+\dots+\left(m_s\pm\frac{\delta}{2a}\right)^2+\dots+m_{d-1}^2\right)+4b^2
\end{array}\]
Since $\delta<a$, this term takes its minimum value precisely when if $m_1=\dots=m_{d-1}=0$, that is, when $x'=x_j$.  For $x'=x_j$ we then obtain
\[ |x_{j+1}-x'|^2 = \delta^{2}+4b^2. \]
Thus $|x_{j+1}-x_j| = \sqrt{\delta^{2}+4b^2}$, as claimed. Hence the lemma follows.
\end{proof}
\smallskip

For an isometry $\varphi$ of $\mathbb{R}^d$ we let $B_\varphi$ denote the matrix of the linear part of $\varphi$, that is, 
\[\varphi(y)=B_{\varphi}y + z_{\varphi}\;\;(y\in\mathbb{R}^d),\] 
for some $z_{\varphi}\in\mathbb{R}^d$.

Our next lemma relates the two sequences of points $(x_j)_{j=-\infty}^{\infty}$ and $(x_j')_{j=-\infty}^{\infty}$ from $X$ which are associated with two point $x$ and $x'$ of $X$, respectively. 

\begin{lemma}\label{2Lclose}
Suppose $0<\delta<a<b$. Let $\varphi$ be an isometry of $\mathbb{R}^d$ that maps point $x\in X_p$ to point $x'\in X_{p'}$ and cluster $C_x(2kR)$ to cluster $C_{x'}(2kR)$ for some integer $k\geq 1$. Then for all $j=-k,\dots,k$,
$$
\varphi(x_j)=
\begin{cases}
x_j',&\text{if  }p-p'\text{ is even},\\
x'_{-j},&\text{if  }p-p'\text{ is odd},
\end{cases}
$$
and
\begin{equation}
\label{XC}
\varphi(X_{p+j}\cap C_x(2kR))=
\begin{cases}X_{p'+j}\cap C_{x'}(2kR),&\text{if  }p-p'\text{ is even},\\
X_{p'-j}\cap C_{x'}(2kR),&\text{if  }p-p'\text{ is odd}.
\end{cases}
\end{equation}
Moreover, $B_{\varphi}e_s=\pm e_{\sigma(s)}$ for $s=1,\ldots,d-1$, for some permutation $\sigma$ of the $\{1,\dots,d-1\}$; and
\begin{equation}
\label{bphi}
B_{\varphi}e_d=
\begin{cases}
e_d,&\text{if  }p-p'\text{ is even},\\
-e_d,&\text{if  }p-p'\text{ is odd}.
\end{cases}
\end{equation}
\end{lemma}

\begin{proof}
First recall that for any $m\in\mathbb{Z}$ and any point $y\in X_m$, the points $y\pm 2ae_s$, with $s=1,\dots,d-1$, belonging to the same layer $X_m$ as $y$ are the points of $X$ closest to $y$. Hence, since $\varphi$ maps $x\in X_p$ to $x'\in X_{p'}$, this shows that $B_\varphi e_{s}=\pm e_{\sigma(s)}$ for some permutation $\sigma$ of $\{1,\dots,d-1\}$. It then follows that $\varphi(X_p)=X_{p'}$. 

The closest point to $x$ in $C_x(2kR)$ not lying in $X_p$ is $x_1$ if $p$ is even, or $x_{-1}$ if $p$ is odd; note here that $R>b$. This immediately implies equation (\ref{bphi}), and therefore $\varphi(X_{p+j})=X_{p'+j}$ or $\varphi(X_{p+j})=X_{p'-j}$ for each $j=-k,\dots,k$, according as $p-p'$ is even or odd. This establishes the second statement of the lemma, that is, equation~{(\ref{XC})}. But then the first statement of the lemma follows as well, by Lemma~\ref{Lclose}, if we can show that $x_j\in C_x(2kR)$ and $x'_j\in C_{x'}(2kR)$ for $j=-k,\dots,k$. Note here that the required condition of Lemma~\ref{Lclose} that $\delta<a$, is guaranteed to hold here by our assumptions on $\delta$ and $a$. 

Now to accomplish the proof of the first statement, we use the defining properties (X1), (X2) and (X3) for the sequences of points $x_j$ and $x_j'$ to estimate the distances $|x_{j}-x_{0}|$ and $|x_{j}'-x_{0}'|$. For $x_j$, we first rewrite $x_j-x_0$ as a sum and use the triangle inequality to obtain
\[ |x_{j}-x_{0}| \;=\; |\sum_{l=1}^{j}(x_{l}-x_{l-1})| \;\leq\; \sum_{l=1}^{j}|x_{l}-x_{l-1}| .\]
Now Lemma~\ref{Lclose} applies to the summands on the right hand side and shows that
%$$
%|x_j-x_0|  \leqslant \left\lceil \frac{k}{2}\right\rceil\sqrt{(2b)^2+\delta^2}+\left\lfloor\frac{k}{2}\right\rfloor(2b)
%< \left\lceil \frac{k}{2}\right\rceil\sqrt{(2b)^2+(d-1)(2a)^2}+\left\lfloor\frac{k}{2}\right\rfloor\sqrt{(2b)^2+(d-1)(2a)^2}
%= 2kR,
%$$
\[\begin{array}{lcl}
|x_j-x_0| & \leqslant &\left\lceil \frac{k}{2}\right\rceil\sqrt{(2b)^2+\delta^2}+\left\lfloor\frac{k}{2}\right\rfloor(2b)\\[.07in]
&<& \left\lceil \frac{k}{2}\right\rceil\sqrt{(2b)^2+(d-1)(2a)^2}+\left\lfloor\frac{k}{2}\right\rfloor\sqrt{(2b)^2+(d-1)(2a)^2}\\[.07in]
&=& 2kR,
\end{array}\]
where $\lfloor\cdot\rfloor$ and $\lceil\cdot\rceil$ denote the floor and ceiling function of a real number respectively. Hence $x_j\in C_x(2kR)$. This establishes the desired property for the points $x_j$. The proof for the points $x_j'$ is similar. This concludes the proof.
\end{proof}
\smallskip

We now investigate symmetries of clusters of Engel sets. Recall that, by definition, the cluster group $S_x(\rho)$ of a $\rho$-cluster $C_x(\rho)$ at a point $x$ of $X$ is the stabilizer of $x$ in the full symmetry group of $C_x(\rho)$. The following theorem describes the cluster groups of $2kR$-clusters of Engel sets for $k\geq 1$. 

Recall that the $d$-{\it crosspolytope\/} (hyperoctahedron) is the $d$-dimensional convex polytope in $\mathbb{R}^d$ with vertices $\pm e_{1},\ldots,\pm e_d$, here viewed as points (see \cite{crp}). When $d=2$ this is a square, and when $d=3$ this is an octahedron. The $d$-crosspolytope is one of the regular solids in $\mathbb{R}^d$. Its full symmetry group is isomorphic to $C_{2}^{d}\ltimes S_d$, the semi-direct product of the elementary abelian group $C_{2}^d$ of order $2^d$, and the symmetric group $S_d$ on $d$ symbols. The subgroup $C_{2}^d$ is generated by the $d$ reflections in the standard coordinate hyperplanes of $\mathbb{R}^d$, and the subgroup $S_d$ consists of all isometries of $\mathbb{R}^d$ that permute the basis vectors $e_{1},\ldots,e_d$. Note that if $L\subseteq \{1,\ldots,d\}$ then $\{\pm e_l\, |\, l\in L\}$ is the vertex set of an $|L|$-dimensional crosspolytope in the $|L|$-dimensional linear subspace spanned by the vectors in $\{e_l\, |\, l\in L\}$. We refer to this polytope as the crosspolytope defined by $\{\pm e_l\, |\, l\in L\}$. For example, if $d=3$ and $L$ has two elements, then this is an equatorial square of the regular octahedron.

\begin{theorem}
\label{Gsym}
Let $d\geq 2$, let $0<\delta<a<b$, let $X:=X(A,a,b,\delta)$, and let $x\in X_p$ for some $p\in \mathbb{Z}$. Further, let $k\geqslant 1$ and $2kR<2b(k+1)$. Then the cluster group $S_x(2kR)$ of the $2kR$-clusters of $X$ is the full symmetry group of the $(d-k-1)$-crosspolytope defined by $\mathcal{E}\setminus \mathcal{E}_{p,k}$ if $1\leqslant k\leqslant d-2$, and is the trivial group if $k\geqslant d-1$. Thus 
\[S_x(2kR)\,\cong\, C_{2}^{d-k-1}\ltimes S_{d-k-1}\;\; (1\leqslant k\leqslant d-2),\]
and $S_x(2kR)=1$ if $k\geqslant d-1$. In particular, this holds when $a^{2}<\frac{2b^2}{k(d-1)}$.
\end{theorem}

\begin{proof}
For the first two statements we may assume without loss of generality that $x=o$, the origin, and hence that $p=0$. Any isometry $\varphi\in S_o(2kR)$ necessarily fixes $o$ and hence is entirely defined by its matrix $B_\varphi$. By Lemma \ref{2Lclose}, applied with $x'=x=o$ and $p=p'=0$, there exists a permutation $\sigma$ of $\{1,\dots,d-1\}$ such that $B_{\varphi}e_s=\pm e_{\sigma(s)}$ for $s=1,\ldots,d-1$, and $Be_d=e_d$. Moreover, again by Lemma \ref{2Lclose}, $Bx_j=x_j$ for $j=-k,\dots,k$ and thus $\varphi(e)=Be=e$ for all vectors $e\in\mathcal{E}_{0,k}$. If follows that $\varphi$ acts trivially on the $(k+1)$-dimensional linear subspace $E_{0,k}$ spanned by the vectors in $\mathcal{E}_{0,k}\cup\{e_d\}$, and that $\varphi$ determines on the orthogonal complement $E_{0,k}^{\perp}$ of $E_{0,k}$ (spanned by the $d-k-1$ basis vectors in $\mathcal{E}\setminus \mathcal{E}_{0,k}$) a symmetry of the $(d-k-1)$-crosspolytope $P_{0,k}$ with vertex set $\{\pm e\,|\, e \in \mathcal{E}\setminus \mathcal{E}_{0,k}\}$ in $E_{0,k}^{\perp}$.

Conversely, any symmetry of the $(d-k-1)$-crosspolytope $P_{0,k}$ in $E_{0,k}^{\perp}$ lifts in an obvious way to an isometry $\varphi$ of $\mathbb{R}^d$ that acts trivially on $E_{0,k}$ and lies in $S_x(2kR)$. Note here that the layers $X_j$ with $|j|\geqslant k+1$ are not involved. This proves the first two statements. 

Finally, note that $a^{2}<\frac{2b^2}{k(d-1)}$ implies $2kR<2b(k+1)$, so that the last statement follows from our previous considerations.
\end{proof}

Theorem~\ref{Gsym} demonstrates nicely how the cluster group of a cluster gets smaller as the cluster grows in size from a radius of $2R$ (for $k=1$), to $2kR$ (for intermediate $k$), to $2(d-2)R$ (for $k=d-2\geq 1$), to $2(d-1)R$ or larger (for $k\geq d-1)$. The corresponding cluster groups of the clusters at these stages are 
\[C_{2}^{d-2}\ltimes S_{d-2},\;\,C_{2}^{d-k-1}\ltimes S_{d-k-1},\;\,C_2,\,\,\mbox{ and }\,\, 1,\]
of orders $2^{d-2}(d-2)!$, $2^{d-k-1}(d-k-1)!$, $2$, and $1$, respectively.
\medskip

Our next theorem characterizes the Engel sets which are regular systems. The upshot is that most Engel sets are not regular systems. The characterization is expressed in terms of the doubly-infinite integer sequence $A$ involved in the construction of Engel sets. 

Recall that condition (A2) for $A$ requires that the terms of $A$ satisfy $|a_{i+d-1}|=|a_i|$ for each $i\in\mathbb{Z}$. Note that there are uncountably many sequences $A$ with the same initial values $a_1,\ldots,a_{d-1}$; in fact, these initial values determine the value of $a_i$ at a position $i\neq 1,\ldots,d-1$ only up to sign, so there are two possible choices at each position. Our next theorem says that among the  uncountable infinity of corresponding Engel sets sharing the same initial values, only two are regular systems.

\begin{theorem}
\label{Xreg}
Let $d\geq 2$ and $0<\delta<a<b$. Then the Engel set $X(A,a,b,\delta)$ is a regular system if and only if the sequence $A$ is such that either $a_{i+d-1}=a_i$ for all $i\in\mathbb Z$ or $a_{i+d-1}=-a_i$ for all $i\in\mathbb Z$.
\end{theorem}

\begin{proof}
First suppose that $X:=X(A,a,b,\delta)$ is a regular system. Then the symmetry group $S(X)$ of $X$ acts transitively on the points of $X$. Choose any point $x\in X$ and consider the sequence of points $x_j$ associated with $x$ as in (X1), (X2) and X(3). Then, since $S(X)$ acts transitively on $X$, there exists an isometry $\varphi\in S(X)$ which maps $x=x_{0}$ to $x':=x_2$ and thus the cluster $C_{x}(2kR)$ to the cluster $C_{x'}(2kR)$ for each $k\geq 1$. Then by Lemma \ref{2Lclose}, $\varphi$~maps $x_{2i}$ to $x_{2i+2}$, and $u_i={\rm sign}(a_i)e_{|a_i|}$ to $u_{i+1}={\rm sign}(a_{i+1})e_{|a_{i+1}|}$, for all $i\in \mathbb{Z}$. Therefore if $\tau=\pm 1$ is defined by the equation $a_d=\tau a_1$ (that is, $\tau=a_{d}/a_{1}$), then also $u_d=\tau u_1$. Hence, $u_{i+d-1}=\tau u_i$ and $a_{i+d-1}=\tau a_i$ for each $i\in\mathbb{Z}$. This establishes one direction of the theorem.

Conversely, suppose the stated condition on $A$ holds. Then $a_{i+d-1}=\tau a_i$ for all $i\in\mathbb Z$, where either $\tau=1$ or $\tau=-1$. Note that then also $u_{i+d-1}=\tau u_i$ for all $i\in\mathbb Z$. We wish to show that this implies that $S(X)$ acts transitively on $X$. Since $X$ is invariant under shifts by vectors from vectors in the $(d-1)$-lattice $Y$ involved in the construction, it suffices to show that we can move the layer $X_0$ of $X$ to any other layer of $X$ by a symmetry $\varphi$ of $X$.

We first explain how layer $X_0$ can be moved to layer $X_2$ and then similarly to any layer $X_{2i}$ for $i\in\mathbb{Z}$. Consider the isometry $\varphi$ obtained as the composition of a shift by the vector~$4be_d$, followed by the isometry $\varphi'$ of $\mathbb{R}^d$ determined by the conditions $\varphi'(e_d)=e_d$ and $\varphi'(u_j)=u_{j+1}$ for $j=1,\ldots,d-1$. We claim that then $\varphi'(u_j)=u_{j+1}$ for all $j\in \mathbb{Z}$. Since $u_{i+d-1}=\tau u_i$ for all $i\in\mathbb Z$, we see that the desired property holds for $j=d$; in fact, $u_d=\tau u_1$ is mapped by $\varphi'$ to $\tau u_2=u_{d+1}$. Now assume inductively that $\varphi(u_j)=u_{j+1}$ holds for a range of consecutive subscripts, $j=s,s+1,\dots, t$ (say) with $|t-s|\geq d$. Then $\varphi'$ takes point $u_{t+1}=\tau u_{t+2-d}$ to point $\tau u_{t+2-d+1}=u_{t+2}$ and similarly point $u_{s-1}=\tau u_{s-1+d-1}$ to point $\tau u_{s-1+d}=u_s$. Therefore, $u_j$ is mapped to $u_{j+1}$ for all $j\in\mathbb{Z}$. It follows that the isometry $\varphi$ is indeed a symmetry of~$X$, that is, $\varphi\in S(X)$.

Next we describe a symmetry of $X$ that moves $X_0$ to $X_1$. We again compose suitable isometries. As in the proof of Lemma \ref{2dm1Lemma}, we first apply the reflection in the hyperplane $x_{d}=b$, in order to map layer $X_m$ of $X$ to layer $X_m'$ of a new layered set $X'$ for each $m\in\mathbb{Z}$, with $X_{1}'=X_0$. After that, we employ the isometry $\varphi'$ of $\mathbb{R}^d$ determined by the conditions $\varphi'(e_d)=e_d$ and $\varphi'(u_j)=-u_{1-j}$ for $j=1,\ldots,d-1$. We claim that then $\varphi'(u_j)=-u_{1-j}$ for each $j\in \mathbb{Z}$. This holds for $j=d$, since 
\[\varphi'(u_d)=\varphi'(\tau u_1)=-\tau u_{0}=-u_{1-d}.\] 
Now assume inductively that $\varphi(u_j)=-u_{1-j}$ holds for a range of consecutive subscripts, $j=s,s+1,\dots, t$ (say) with $|t-s|\geq d$. Then the isometry $\varphi'$ takes point $u_{t+1}=\tau u_{t+2-d}$ to point
\[-\tau u_{d-t-1}=-u_{1-(t+1)},\] 
and point $u_{s-1}=\tau u_{s-1+d-1}$ to point 
\[-\tau u_{3-s-d}=-u_{2-s}=-u_{1-(s-1)}.\] 
Thus, $\varphi'$ maps $u_j$ to $-u_{1-j}$ for all $j$. It follows that the composed isometries give a symmetry of $X$ which takes $X_0$ to $X_1$. This completes the proof.

%Next we describe a symmetry of $X$ that moves $X_0$ to $X_1$. We again compose suitable isometries. As in the proof of Lemma \ref{2dm1Lemma}, we first apply the reflection in the hyperplane $x_{d}=0$, in order to map layer $X_m$ of $X$ to layer $X_m'$ of a new layered set $X'$ for each $m\in\mathbb{Z}$, with $X_{0}'=X_0$, and then shift by the vector $2be_d$. Lastly, we employ the isometry $\varphi'$ of $\mathbb{R}^d$ determined by the conditions $\varphi'(e_d)=e_d$ and $\varphi'(u_j)=-u_{1-j}$ for $j=1,\ldots,d-1$. We claim that then $\varphi'(u_j)=-u_{1-j}$ for each $j\in \mathbb{Z}$. This holds for $j=d$, since 
%\[\varphi'(u_d)=\varphi'(\tau u_1)=-\tau u_{0}=-u_{1-d}.\] 
%Now assume inductively that $\varphi(u_j)=-u_{1-j}$ holds for a range of consecutive subscripts, $j=s,s+1,\dots, t$ (say) with $|t-s|\geq d$. Then the isometry $\varphi'$ takes point $u_{t+1}=\tau u_{t+2-d}$ to point
%\[-\tau u_{d-t-1}=-u_{1-(t+1)},\] 
%and point $u_{s-1}=\tau u_{s-1+d-1}$ to point 
%\[-\tau u_{3-s-d}=-u_{2-s}=-u_{1-(s-1)}.\] 
%Thus, $\varphi'$ maps $u_j$ to $-u_{1-j}$ for all $j$. It follows that the composed isometries give a symmetry of $X$ which takes $X_0$ to $X_1$. This completes the proof.
\end{proof}

The previous theorem allows us to conclude that there are only few regular systems among Engel sets. More precisely, we have the following corollary.

\begin{corollary}\label{cor:regularity}
Let $d\geq 2$ and $0<\delta<a<b$. Then among the uncountably many Engel sets $X(A,a,b,\delta)$ obtained for sequences $A$ with the same initial values $a_1,\ldots,a_{d-1}$ there are up to isometry exactly two (different) regular systems, namely those obtained for the two sequences $A$ with $a_{i+d-1}=a_i$ for all $i\in\mathbb Z$ or $a_{i+d-1}=-a_i$ for all $i\in\mathbb Z$ (that is, those corresponding to $\tau=1$ and $\tau=-1$, respectively).
\end{corollary}

\begin{proof}
Theorem~\ref{Xreg} implies that there can be at most two regular systems of the kind described. It remains to show that the regular systems corresponding to $\tau=1$ and $\tau=-1$ cannot be congruent.

Suppose the regular systems $X(A,a,b,\delta)$ and $X(A',a,b,\delta)$ obtained for two sequences $A$ and $A'$ (with the same initial values $a_1,\ldots,a_{d-1}$) are congruent under an isometry $\varphi$ of~$\mathbb{R}^d$. Without loss of generality we may assume that $\varphi$ fixes the origin and thus is a linear isometry. Let $u_i$ and $u_{i}'$ denote the corresponding vectors as defined in (E4). Then a similar argument as in the proof of Lemma~\ref{2Lclose} shows that we must have $\varphi(u_i)=u_i'$ for all $i\in\mathbb{Z}$. Hence, if $\kappa=\pm 1$ is defined by the equation $u_d=\kappa u_1$ (that is, $\kappa=u_{d}/u_1$), then we must also have $u_d'=\kappa  u_1'$.

It follows immediately that the two regular systems corresponding to $\tau=1$ and $\tau=-1$ cannot be congruent.
\end{proof}

Combining the previous corollary with Theorem~\ref{thm:onecluster} we finally obtain the main result of this paper.

\begin{theorem}
\label{thm:2dR} 
Suppose $d\geq 2$ and $R$ is a fixed positive number. For any $\varepsilon$, with $0<\varepsilon<2dR$, there exists a non-regular Delone set $X$ of type $(r,R)$ in $d$-space such that $N_X(2dR-\varepsilon)=1$.
\end{theorem}

\begin{proof}
Choose a pair of parameters $a,b$ such that 
\begin{equation}
\label{crucial}
a<b,\;\, R=\sqrt{b^2+(d-1)a^2},\;\, \mbox{and}\;\, a^2<\frac{\varepsilon b}{d(d-1)},
\end{equation} 
and choose the initial part of the doubly-infinite sequence $A$ as $a_1,\ldots,a_{d-1}$ (but allow the signs of the other terms $a_i$ of $A$ to vary). We also choose a number $\delta$ such that $0<\delta<a$. Then Theorem~\ref{thm:onecluster} implies that any Engel set $X$ from the infinite family of Engels sets $X(A,a,b,\delta)$ satisfies the condition $N_X(2dR-\varepsilon)=1$. However, according to Corollary~\ref{cor:regularity}, only two of these sets will be regular sets.
\end{proof}

Now we can also establish the new lower bound for the regularity radius of Delone sets. Recall that, given the parameters $r$ and $R$, the regularity radius $\hat{\rho}_d$ is the smallest positive number $\rho$ with the property that each Delone set $X$ of type $(r,R)$ in $\mathbb{R}^d$ with mutually equivalent $\rho$-clusters is a regular point system. Then the previous theorem immediately implies the following lower bound for $\hat{\rho}_d$ (when $d\geq 2$), which is linear in $d$. But the bound is also valid for $d=1$.

\begin{theorem}\label{cor:lowerbound}
For $d\geq 1$ we have $\hat{\rho}_d\geq 2dR$.
\end{theorem}

\begin{proof}
As we saw at the end of Section~\ref{bano}, the inequality holds when $d=1$. Now suppose $d\geq 2$ and $0<\rho<2dR$. Set $\varepsilon:=2dR-\rho$ and apply Theorem~\ref{thm:2dR}. The Delone set $X$ of Theorem~\ref{thm:2dR} then satisfies $N_X(\rho)=1$ but is not a regular system. This shows that no positive number $\rho$ less than $2dR$ has the property that each Delone set of type $(r,R)$ with mutually equivalent $\rho$-clusters is a regular point system. Thus $\hat{\rho}_d\geq 2dR$.
\end{proof}

\begin{remark}
The Engel sets from the proofs of Theorems \ref{thm:2dR} and \ref{cor:lowerbound} are characterized by $b$ much larger than $a$, hence they are very elongated in one direction. From a crystallographic point of view, this would imply unphysically large distance between consecutive layers compared to distances between points within one layer. On the other hand, it is possible to introduce more realistic Engel sets with $a$ close to $b$ and a small $\delta$. These sets can be obtained by a slight deformation of a regular cubic lattice $2a\mathbb Z^d$.  For such a Delone set $R$ is close to $\sqrt{d}a$ and (due to Lemma \ref{2dm1Lemma}) the $\rho$-clusters with $\rho$ close to $2da$ are congruent. By taking  an appropriate sequence $A$ we obtain a non-regular Delone set $X(A,a,b,\delta)$. Thus, for such deformations of a regular lattice the congruence of $(2\sqrt{d}R-\varepsilon)$-clusters is not enough to ensure regularity.
\end{remark}

Our last theorem characterizes the regularity of an Engel set in terms of its cluster counting function. In particular it says that, for Engel sets, the radius $2dR-\varepsilon$ in Theorem~\ref{thm:2dR} cannot be replaced by $2dR$ (or any larger number). 

\begin{theorem}
\label{enreg}
Let $d\geq 2$ and $0<\delta<a<b$.  Then an Engel set $X:=X(A,a,b,\delta)$ is a regular system if and only if $N_{X}(2dR)=1$.
\end{theorem}

\begin{proof}
Clearly, if $X$ is a regular system, then $N_{X}(\rho)=1$ for each $\rho>0$ and thus $N_{X}(2dR)=1$. Conversely, suppose $N_{X}(2dR)=1$. Take $x:=o$ and $x':=x_{2i}$ for some $i\in\mathbb{Z}$. Then, by Lemma \ref{2Lclose} with $k=d$, the linear part of the isometry that realizes the equivalence of the clusters $C_x(2dR)$ and $C_{x'}(2dR)$ must map $u_l$ to $u_{i+l}$ for each of the $d$ values $l_0,\dots,l_0+d-1$ (say) of $l$ arising in the construction of the layers $X_{-d},\dots,X_d$ of $X$. If $\tau=\pm 1$ is defined by $u_{l_0+d-1}=\tau u_{l_0}$ (that is, $\tau:=u_{l_0+d-1}/u_{l_0}$), then necessarily $u_{i+d-1}=\tau u_i$ for all $i\in\mathbb{Z}$. Hence $X(A,a,b,\delta)$ is a regular system.
{\sloppy

}
\end{proof}

\section{Engel sets in crystallography and crystal chemistry}

The example of 3-dimensional Engel sets with identical $4R$ but different $6R$ environments is shown in Figure \ref{engel-fig4}. The first set (Figure \ref{engel-fig4}a) is a crystallographic orbit in the space group $P4_122$ with the approximate unit-cell parameters $a = 3, c = 25$ \AA, and point coordinates (0.200; 0.200; 0.420). The second set (Figure \ref{engel-fig4}b) is an orbit of the point (0.029; 0.080; 0.340) in the space group $C2/c$ with the unit-cell parameters $a = b = 4.245, c = 12.6$ \AA, $\beta$ = $96^\circ$. Both orbits are Delone sets with $r \approx 1.125$ and $R \approx 2.795$ \AA. They are also Engel sets as defined above: they can be considered as consisting of double layers of 2-dimensional square lattice (shown by black lines).

\begin{figure}
\begin{center}
\includegraphics[width=\textwidth]{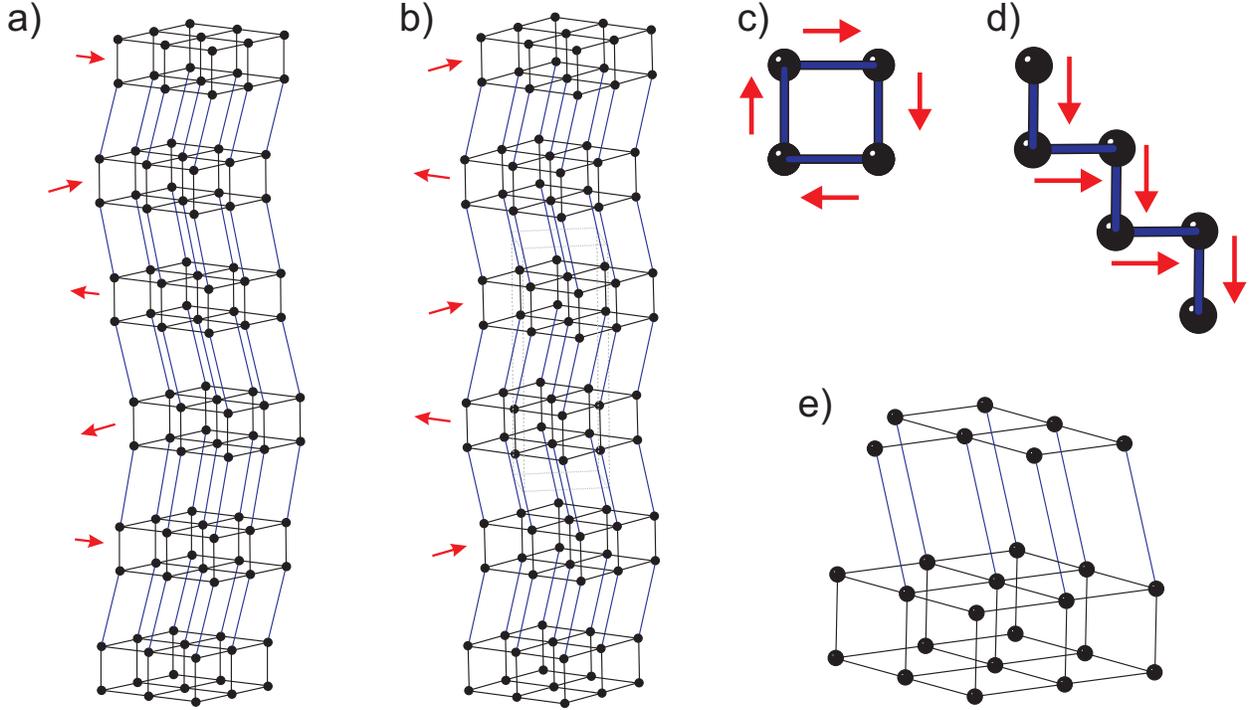}
\end{center}
\caption{Two examples of the Engel sets in 3-space with identical $4R$- and different $6R$-clusters: the set with the $P$4\(_1\)22 space group (a), the set with the $C2/c$ space group (b) (the black lines outline the double layers, the blue lines link the closest points from the adjacent double layers, red arrows indicate the direction of shifts of double layers relative to each other), the top views on the two sets (c and d, respectively), showing the systems of shifts of adjacent rectangular double layers, and $2R$-cluster of a single point in the sets (e; identical for the two sets).}
\label{engel-fig4}
\end{figure} 

The $P4_122$ structure contains four double layers per unit cell with each double layer shifted relative to the adjacent ones by 1.2 \AA, along either $a$ or $b$ axes. The sequence of shifts in the $P4_122$ structure can be described as $...+a_i; +b_i; -a_i; -b_i...$, where $\vert a_i \vert = \vert b_i \vert$ = 1.2 \AA (Figure \ref{engel-fig4}c).

The $C2/c$ structure contains two double layers per unit cell only with the shifts of adjacent layers along the [110] and [-110] directions. The sequence of shifts can be described as $...a_i+b_i; b_i-a_i; ...$, where $\vert a_i+b_i \vert = \vert b_i-a_i \vert$ = 1.2 \AA (Figure \ref{engel-fig4}d).

The two sets have congruent $2R$ (Figure \ref{engel-fig4}e) and $4R$ environments  of their points, but are obviously different. The $2R$ environment of the point defines the double layer to which it belongs plus one square lattice from the adjacent double layer. The $4R$ environment defines the positions of the two adjacent double layers, which is obviously not enough to define the whole structure, that can be fixed only by fixing the position of the next-neighboring layer through the $6R$ environment as shown in Figure \ref{engel-fig5}. 

\begin{figure}
\begin{center}
\includegraphics{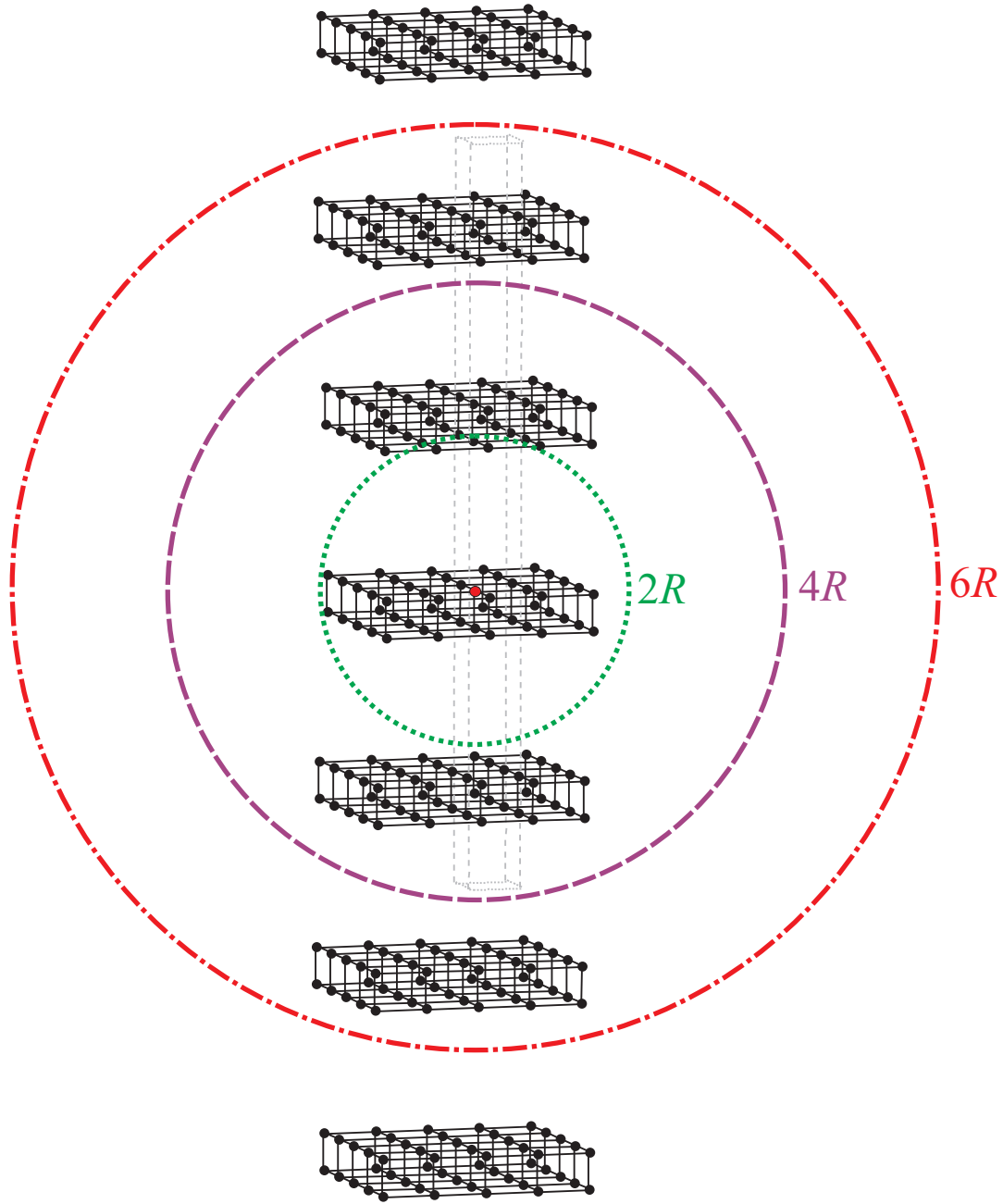}
\end{center}
\caption{Schematic representation of an Engel set with the spheres of the $2R$-, $4R$- and $6R$-radii (showing by green dotted, violet dash and red dot-dash lines, respectively) drawn around a central point (shown in red). See text for details.} 
\label{engel-fig5}
\end{figure}

The two sets shown in Figure \ref{engel-fig4} provide a crystallographic example of the Theorem \ref{enreg}, for which precise and rigorous mathematical proof is given above. From the standpoints of crystal chemistry and crystallography, it is obvious that the Engel sets can be viewed as models of polytypic structures, which are very common in layered materials. According to \cite{guinier}, ``...an element or compound is polytypic if it occurs in several different structural modifications, each of which may be regarded as built up by stacking layers of (nearly) identical structure and composition, and if the modifications differ only in their stacking sequence.'' The two structures shown in Figure \ref{engel-fig4} can be considered as classical examples of polytypes with the $P4_122$ structure being $4T$- and the $C2/c$ structure being $2M$-polytypes. The formation of long-range order arrangements (ideal crystals) in this family of structures would require the fullfillment of the $6R$ regularity radius requirement, whereas violation of this condition would result in the formation of disordered layer stackings. Thus Theorem \ref{enreg} is not only an important fundamental result, but also provides a useful tool for understanding of geometrical conditions for the formation of ordered and disordered polytypic layered materials. Finally, we would like to point out that the Theorem \ref{enreg} provides the exact lower bound of $6R$ in 3-dimensional Euclidean space, whereas $10R$ is the currently proved upper bound. It is very likely that $6R$ is an exact upper bound, however, there is no proof for this statement at the moment. 

We believe that the local theory first proposed in \cite{dedostga} and developed herein will be useful for the understanding mechanisms of crystallization of complex crystalline materials at the micro- and nanoscopic levels and the development of theories of self-assembly in inorganic and molecular systems.

     %-------------------------------------------------------------------------
     % The back matter of the paper - acknowledgements and references
     %-------------------------------------------------------------------------

     % Acknowledgements come after the appendices

\section*{Acknowledgements}

The work of Nikolay Erokhovets was supported in part by Young Russian Mathematics award. The work of Egon Schulte was partially supported by Simons Foundation award no. 420718. Sergey Krivovichev was supported in this work by the President of Russian Federation grant for leading scientific schools (grant NSh-3079.2018.5).

%% REFERENCES

%\providecommand{\bysame}{\leavevmode\hbox to3em{\hrulefill}\thinspace}
%\providecommand{\MR}{\relax\ifhmode\unskip\space\fi MR }
%\providecommand{\MRhref}[2]{%
  %\href{http://www.ams.org/mathscinet-getitem?mr=#1}{#2}
%}
%\providecommand{\href}[2]{#2}
%
%
%\bibliographystyle{iucr}
%
%

\end{document}